\documentclass[11pt,a4paper]{amsart}
\usepackage[utf8]{inputenc}
\usepackage{amsmath}
\usepackage{amsthm}
\usepackage{import}
\usepackage{geometry}
\usepackage{mathrsfs} 
\usepackage{amsfonts}
\usepackage{amssymb}
\usepackage{palatino}
\usepackage{hyperref}
\usepackage{float}
\usepackage[font=small]{caption}
\usepackage{comment,accents}
\usepackage[export]{adjustbox}
\usepackage{enumerate,graphicx}
\usepackage[usenames]{color}

\newtheorem{mydef}{Definition}[section]

\newtheorem{theorem}{Theorem}[section]
\newtheorem{lemma}{Lemma}[section]

\newtheorem{remark}{Remark}[section]

\newtheorem{assumptions}[theorem]{Assumptions}

\newtheorem{cor}{Corollary}[section]

\newcommand{\fishu}{\scrI_U}

\newcommand{\entab}{\scrH_{U}}

\newcommand{\id}{\mathbf{id}}
\newcommand{\otgrad}{\nabla^{\mathcal{W}}}

\newcommand{\otcovcur}{\frac{\bD}{dt}}
\newcommand{\fish}{\mathscr{I}_U}

\newcommand{\vol}{\mathbf{vol}}

\newcommand{\tsp}{\mathbf{T}}
\newcommand{\ent}{\scrH_U}
\newcommand{\ric}{\mathbf{\mathfrak{Ric} } }
\newcommand{\entcost}{ \cT^{\sigma}_{U} }
\newcommand{\fwh}{h_f}
\newcommand{\bwh}{h_b}
\newcommand{\divg}{\mathbf{div}}
 
\newcommand{\inv}{\mathbf{m}}
\newcommand{\supp}{\text{supp}}
\newcommand{\hess}{\mathbf{Hess}^{\mathcal{W}}}

\newcommand{\cinfp}{\cC^{+}_{\infty}}
\newcommand{\cinf}{\cC_{\infty}}
\newcommand{\cinfc}{\cC^{b}_{\infty}}
\newcommand{\cinfcp}{\cC^{b,+}_{\infty}}
\newcommand{\cinfcs}{\cC^{c}_{\infty}}

\newcommand{\hQ}{\hat{\bQ}}
\newcommand{\levciv}{\overline{\nabla}}

\newcommand{\R}{\mathbb{R}}

\newcommand{\RD}{\mathbb{R}^d}
\newcommand{\N}{\mathbb{N}}

\newcommand{\cC}{\mathcal{C}}

\newcommand{\cP}{\mathcal{P}}

\newcommand{\cX}{\mathcal{X}}

\newcommand{\cT}{\mathcal{T}}

\newcommand{\tf}{\tilde{f}}
\newcommand{\tg}{\tilde{g}}

\newcommand{\scrE}{\mathscr{E}}
\newcommand{\scrF}{\mathscr{F}}
\newcommand{\scrH}{\mathscr{H}}
\newcommand{\scrI}{\mathscr{I}}
\newcommand{\scrU}{\mathscr{U}}

\newcommand{\scrL}{\mathscr{L}}

\newcommand{\bes}{\begin{equation*}}
\newcommand{\ees}{\end{equation*}}
\newcommand{\beas}{\begin{eqnarray*}}
\newcommand{\eeas}{\end{eqnarray*}}
\newcommand{\bea}{\begin{eqnarray}}
\newcommand{\eea}{\end{eqnarray}}
\newcommand{\be}{\begin{equation}}
\newcommand{\ee}{\end{equation}}

\newcommand{\veps}{\varepsilon}

\newcommand{\bd}{\mathbf{d}}
\newcommand{\bD}{\mathbf{D}}

\newcommand{\bT}{\mathbf{T}}
\newcommand{\bP}{\mathbf{P}}
\newcommand{\bQ}{\mathbf{Q}}

\newcommand{\bbl}{\begin{block}}
\newcommand{\ebl}{\end{block}}

\usepackage{blindtext}
\usepackage{hyperref}
\usepackage{nameref}

\newcounter{mylabelcounter}

\makeatletter
\newcommand{\labelText}[2]{%
#1\refstepcounter{mylabelcounter}%
\immediate\write\@auxout{%
  \string\newlabel{#2}{{1}{\thepage}{{\unexpanded{#1}}}{mylabelcounter.\number\value{mylabelcounter}}{}}%
}%
}
\makeatother
\author{Giovanni Conforti}

\address{D\'epartement de Math\'ematiques Appliqu\'ees, \'Ecole Polytechnique,
Route de Saclay, 91128, Palaiseau Cedex, France}
\email{giovanniconfort@gmail.com}

\begin{document}
\title[Schr\"odinger bridges, hot gas experiment and entropic transportation cost]{A second order equation for Schr\"odinger bridges with applications to the hot gas experiment and entropic transportation cost}

\begin{abstract}
	The \emph{Schr\"odinger problem} is obtained  by replacing the mean square distance with the relative entropy in the Monge-Kantorovich problem. It was first addressed by  Schr\"odinger as the problem of describing the most likely evolution of a large number of Brownian particles conditioned to reach an ``unexpected configuration". Its optimal value, the \textit{entropic transportation cost}, and its optimal solution, the \textit{Schr\"odinger bridge}, stand as the natural probabilistic counterparts to the transportation cost and displacement interpolation.  Moreover, they provide a natural way of lifting from the point to the measure setting the concept of Brownian bridge. In this article, we prove that the Schr\"odinger bridge solves a second order equation in the Riemannian structure of optimal transport. Roughly speaking, the equation says that its acceleration is the gradient of the Fisher information. Using this result, we obtain a fine quantitative description of the dynamics, and a new functional inequality for the entropic transportation cost, that generalize Talagrand's transportation inequality.
Finally, we study the convexity of the Fisher information along Schr\"odigner bridges, under the hypothesis that the associated \textit{reciprocal characteristic} is convex. The techniques developed in this article are also well suited to study the  \emph{Feynman-Kac penalisations} of Brownian motion.
\end{abstract}
\maketitle
\section{Introduction and statement of the main results}\label{sec:1}
 The \emph{Schr\"odinger problem}  \eqref{eq:SPs} is the problem of finding the best approximation in the relative entropy sense of a stationary dynamics $\bP$  under constraints  on the marginal laws. It originates from the early works \cite{Schr32,Schr} and is now an active field of research with ties to many others. Without being exhaustive, we mention  \cite{Foll88,follmer1997entropy,cattiaux1994minimization,dawson1990schrodinger,dawson1994multilevel} for connections with large deviations and \cite{DP91,Mik04,benamou2015iterative,leonard2012schrodinger,gentil2015analogy,chen2014relation,Wak89,gigli2017second}
 for the relations with optimal transport and stochastic control. In \cite{Kre88,Kre97,LeKre93,Th93,CruZa91,Zam86} Schr\"odinger bridges are used to develop Euclidean Quantum Mechanics (EQM) and second order calculus for diffusions, while
 in the article \cite{leonard2016lazy} they are employed to construct interpolations between probability measures on discrete spaces. We refer to \cite{chen2016optimalA,chen2016optimalB,Galichon2016Schr,solomon2015convolutional,li2017computations} for applications in engineering, economics and graphics.  In this article we prove that the \emph{Schr\"odinger bridge} $\hat{\bQ}$, i.e. the solution to \ref{eq:SPs},
solves a second order differential equation in the Riemannian structure of optimal transport and  obtain quantitative results for its dynamics. In particular, since the evolution of the empirical measure of $N$ Brownian particles whose initial and final configurations are known is described in the limit for $N\rightarrow + \infty$ by a Schr\"odinger bridge, our results  bring answers to the basic question (see also section \ref{sub:hotgas} )
\begin{itemize}
\item[\labelText{\textbf{Q}}{secondquest}] Suppose you observe the configuration at times $t=0,1$ of $N \gg 1$ independent Brownian particles. How well does their configuration at time $t\in (0,1)$  resemble the equilibrium configuration?
\end{itemize}

Moreover, we shall establish a new connection between the so called \emph{reciprocal characteristics} associated with a potential, and the convexity of the Fisher information along Schr\"odinger bridges. Although the main objective of this paper is to understand the dynamics of Schr\"odinger bridges from a probabilistic viewpoint, our results can be seen as ``stochastic" generalizations of well known results of optimal transport, such as Talagrand's inequality and displacement convexity of the entropy. Indeed, these classical results can be recovered through a small noise argument.

\subsection*{Organization of the paper}   
The paper is organized as follows: in section \ref{sec:1} we state the results; in section \ref{sec:2} we recall some notions of second order calculus in Wasserstein space, and in section \ref{sec:3} we provide proofs. Some technical lemmas are in the Appendix.

 \subsection{The Schr\"odinger problem}\label{sub:hotgas}
In this subsection we give an introduction to \ref{eq:SPs} and collect some useful facts.
 To formulate \ref{eq:SPs}, let $\Omega$ be the space of continuous paths over the time interval $[0,1]$ with values in a smooth complete connected Riemannian Manifold $(M,g)$ without boundary. We consider the law $\bP$ on $\Omega$ of the only stationary Markov measure\footnote{$\bP$ may not be a probability measure, but an infinite measure. This won't be a problem as long as $\bQ$ is a probability measure. The note \cite{leonard2014some} takes care of this issue in detail.
} for the generator
\bes
\scrL  = \sigma \Big( \frac{1}{2} \Delta - \nabla U \cdot \nabla \Big),
\ees
where $\sigma$ is a positive constant, $\Delta$ the Laplace-Beltrami operator and $U$ a smooth Lipschitz potential. 
The Schr\"odinger problem is the problem of finding the best approximation of $\bP$ in the relative entropy sense in the set of probabilities with prescribed marginals $\mu,\nu$  at times $t=0,1$.
 \be\label{eq:SPd}\tag{SPd}  \nonumber \min \scrH(\bQ | \bP) \quad 
 \bQ \in \cP(\Omega)\,, {X_0}_{\#} \bQ =\mu, {X_1}_{\# } \bQ= \nu ,
 \ee
 where $\cP(\Omega)$ is the space of probabilities over $\Omega$, $X_t(\omega) = \omega_t$ is the canonical projection map, $\#$ the push-forward and $\scrH( \cdot |\bP)$ the relative entropy functional.
 We refer to this formulation as the \emph{dynamic formulation}. The optimal measure  $\hat{\bQ}$ is  the \emph{Schr\"odinger  bridge}  \labelText{SB($\scrL,\mu,\nu$)}{schrbr} between $\mu$ and $\nu$.
The \emph{static formulation} is obtained by projecting onto the endpoint marginals. 
   \be\label{eq:SPs}\tag{SP}  \nonumber \min \scrH(\pi | (X_0,X_1)_{\#} \bP), \quad 
 \pi \in \Pi(\mu,\nu),  \ee
 where $\Pi(\mu,\nu)$ is the set of couplings of $\mu$ and $\nu$.
We call the optimal value $\entcost(\mu,\nu)$ of \ref{eq:SPd} the \emph{entropic transportation cost}. It is known \cite[Prop. 2.4]{LeoSch} that an optimal solution $\hat{\bQ}$ of \ref{eq:SPd} exists whenever $\entcost(\mu,\nu)<+\infty$ and that the two formulations are equivalent, see \cite{Foll88}: the Schr\"odinger bridge $\hat{\bQ} $ is obtained lifting the optimal coupling $\hat{\pi}$ of \ref{eq:SPs} to path space by mean of bridges. We have
\be\label{eq:bridgedec} \hat{\bQ}(\cdot ) = \int_{M \times M} \bP^{xy}(\cdot) \hat{\pi}(dx \, dy),  \ee 
where $\bP^{xy}$ is the \emph{xy bridge} of $\bP$ (\cite{LeoSch} for details). Using this decompositon, one can show that $\entcost(\mu,\nu)$ is also the optimal value of \ref{eq:SPs}, and therefore that the optimal values of \ref{eq:SPd} and \ref{eq:SPs} coincide. 

\subsubsection*{The $\emph{fg}$ representation of the optimal solution}
It has been shown at \cite[Th 2.8]{LeoSch} that, under some mild assumptions on $\mu,\nu$ and $\bP$, there exist two measurable functions $f,g: M \rightarrow \R_{\geq 0}$ such that $E_{\bP}(f(X_0)g(X_1))=1$ and the SB takes the form
\be\label{eq:fgdec}  \hat{\bQ} = f(X_0)g(X_1) \bP. \ee
 $f,g$ are found by solving the \emph{Schr\"odinger system}, see \cite{LeoSch},\cite{rus93}. In the rest of the article we will often deal with the functions $f_t,g_t$ defined for $t \in [0,1]$, by
\be\label{def:fg} f_t(x) := E_{\bP}[f(X_0)|X_t=x], \quad g_t(x) := E_{\bP}[g(X_1)|X_t=x] .   \ee

The following ``hot gas experiment" provides an heuristic motivation for the formulation of \ref{eq:SPs}, and helps building intuition.
\subsubsection*{Hot gas experiment}
 At time $t=0$ we are given $N$ independent particles  $(X^i_t)_{t \leq 1, i \leq N}$ whose configuration $\mu$ is
\[ \mu := \frac{1}{N} \sum_{i=1}^N \delta_{X^i_0} \quad .\] 
We then let the particles travel independently following the Langevin dynamics for $\scrL$ and look at their configuration $\nu$ at $t=1$,
 \[ \nu:= \frac{1}{N}\sum_{i=1}^N \delta_{X^i_1} \quad . \]
If the Langevin dynamics has good ergodic properties and $N$ is very large, one expects $\nu$ to be very close to the solution at time $1$ of the Fokker-Planck equation associated with the generator $\scrL$ and the initial datum $\mu$. However, although very unlikely, it is still possible to observe an \emph{unexpected configuration}\footnote{Schr\"odinger writes in \cite{Schr32} \emph{``un \'ecart spontan\'e et considerable par rapport \`a cette uniformit\'e" }} . Schr\"odinger's question is to find the most likely evolution of the particle system conditionally on the fact that such a rare event happened. Letting $N \rightarrow + \infty$ and using Sanov's  Theorem, this question can rigorously be formulated as \ref{eq:SPd} (see also \cite[sec.6]{LeoSch}). The Schr\"odinger problem can be viewed as a stochastic counterpart to the ``lazy gas experiment" of optimal transport \cite[Ch. 16]{villani2008optimal}. Indeed
\begin{itemize}
	\item Particles choose their final destination minimizing the relative entropy instead of the mean square distance.
	\item Particles travel along Brownian bridges instead of geodesics, see \eqref{eq:bridgedec}.
\end{itemize}
We refer to \cite[sec. 6]{leonard2013convexity} for an extensive treatment of this analogy.
Quite remarkably, the description of the hot gas experiment we gave is very similar to the original formulation of the problem that Schr\"odinger proposed back in 1932, (see \cite{Schr32}) when the modern tools of probability theory were not available.

\subsection{An equation for the Schr\"odinger bridge.}
 The first result of this article is that the marginal flow $(\mu_t)$ of \nameref{schrbr} solves a second order differential equation, when viewed as a curve in the space $\cP_2(M)$ of probability measures with finite second moment endowed  with the Riemannian-like structure of optimal transport, \cite{otto2001geometry,sturm2006geometry,lott2009ricci}. Let us, for the moment, provide the minimal notions to state the result; we  postpone to Section 2 a summary based on \cite{ambrosio2013user,giglisecond} of second order calculus in $\cP_2(M)$. As said, we view $(\mu_t)$ as a curve in a kind of Riemannian manifold. Therefore, provided it is regular enough, we can speak of its velocity and acceleration. If $(\mu_t)$ is a \textit{regular curve} (Def.\ref{def:reg}), starting from the continuity equation
\[\partial_t \mu_t + \nabla \cdot(\mu_t v_t)=0, \]
 one can define a Borel family of vector fields, which is the \textit{velocity field} of $(\mu_t)$. Furthermore, if $v_t$ is \textit{absolutely continuous} along $(\mu_t)$ (Def \ref{def:smooth}) we can compute its \textit{covariant derivative} (Def. \ref{def:covder}) $\frac{\bD}{dt} v_t$, which plays the role of the acceleration of $(\mu_t)$. We consider the \textit{Fisher information} functional $\fish$ on $\cP_2(M)$
\be\label{eq:fishdef} \fish(\mu) = \begin{cases} \int_M |\nabla (\log \mu + 2 U) |^2 d \mu \quad & \mbox{if $\mu \ll \vol $} \\ + \infty \quad & \text{otherwise}  \end{cases}\ee
In the definition above, and in the rest of the paper we make no distinction between a measure $\mu$ and its Radon-Nykodym density $\frac{d \mu}{ d \vol}$ w.r.t. the volume measure $\vol$ on $M$. Moreover, if $v$ is a vector field on $M$, we abbreviate $|v(z)|_{\mathbf{T}_zM}$ with $|v|$. In Definition \ref{def:otgrad} we provide the definition of gradient of a functional over $\cP_2(M)$. According to this notion, the gradient of $\fish$ is the vector field
\be\label{eq:fishgrad} \otgrad \fish (\mu) := \nabla \Big( - |\nabla \log \mu|^2 - 2 \Delta \log \mu   + 8 \scrU \Big) , \quad \text{where} \quad \scrU = \frac{1}{2}(|\nabla U|^2- \Delta U).
\ee
 The content of the next Theorem is that under some suitable regularity assumptions, the acceleration 
of $(\mu_t)$ is $\otgrad \fish (\mu_t)$. 
 \begin{assumptions}\label{ref:mainass}
$M,U,\mu,\nu$ satisfy one of the following 
\begin{enumerate}[(A)]
\item  $M$ is compact, $U \in \cC_{\infty}(M)$, $\ent(\mu),\ent(\nu)<+\infty $. 
\item $M =\RD$, $U(z) = \frac{\alpha}{2} |z|^2  $ for some $\alpha>0$, $\mu,\nu \ll \vol $ have compact support and bounded density against $\vol$. 
\end{enumerate}
\end{assumptions}
\begin{theorem}\label{thm:2ndorderODE}
Let $M,U,\mu,\nu$ be such that Assumption \ref{ref:mainass} holds. Then the marginal flow $(\mu_t)$ of \nameref{schrbr} is a regular curve and its velocity field $(v_t)$ is absolutely continuous.  Moreover, $(\mu_t)$ satisfies the equation
\be\label{eq:2ndorderODE}
\forall \, 0<t<1, \quad	\frac{\bD}{dt}  v_t = 	 \frac{\sigma^2}{8}\otgrad  \fish 	(\mu_t),
\ee
where $\frac{\bD}{dt} v_t$ is the \emph{covariant derivative} of $(v_t)$ along $(\mu_t)$.
\end{theorem}
Assumption \ref{ref:mainass} essentially says that either $M$ is compact or $\bP$ is a stationary Ornstein-Uhlenbeck process on $\RD$. We impose this strong assumptions because we do not want to assume any regularity on the solution to \ref{eq:SPd}, but only on the data of the problem. If we accept to do so, Assumption \ref{ref:mainass} can be dropped. 
\begin{theorem}\label{thm:2ndorderODEb}
Assume that the dual representation \eqref{eq:fgdec} of \nameref{schrbr} holds, that the marginal flow $(\mu_t)$ is such that 
\be\label{eq:2ndorderODEb1}  \forall \, \veps \in (0,1) \quad \sup_{t \in [\veps,1-\veps] } |\otgrad \fish(\mu_t)|_{L^2_{\mu_t}}<+\infty,  \ee
and that $f_t,g_t$ are such that for all $\veps \in (0,1)$
\be\label{eq:2ndorderODEb2}  \sup_{t \in  [\veps,1-\veps]} |\nabla (\log g_t-\log f_t)|_{L^2_{\mu_t}}<+\infty ,\quad \sup_{\substack{t \in  [\veps,1-\veps]\\ x \in M}} |\mathbf{Hess}_x (\log g_t - \log f_t) |_{op}<+\infty. \ee
Then the marginal flow $(\mu_t)$  is a regular curve and its velocity field $(v_t)$ is absolutely continuous. Moreover, $(\mu_t)$ satisfies the equation
\be\label{eq:2ndorderODEb}
\forall \, 0<t<1, \quad	\frac{\bD}{dt}  v_t = 	 \frac{\sigma^2}{8} \otgrad  \fish (\mu_t),
\ee
where $\frac{\bD}{dt} v_t$ is the \emph{covariant derivative} of $(v_t)$ along $(\mu_t)$.   
\end{theorem}

 After some manipulations, it is possible to reinterpret the fluid dynamic formulations of \ref{eq:SPd} obtained in
   \cite{gentil2015analogy} and \cite{chen2014relation} as variational problems in the Riemannian manifold of optimal transport: equation \eqref{eq:2ndorderODEb} is then the associated Euler-Lagrange equation. Thus, in principle we could link \eqref{eq:2ndorderODEb} with the theory of Hamiltonian system in $\cP_2(M)$, see \cite{ambrosio2008hamiltonian}. However, our proof is based on probabilistic arguments, and does not use that framework.   
By changing the sign in the right hand side of \eqref{eq:2ndorderODE}, one gets a nice connection with the Schr\"odinger equation, see \cite{von2012optimal,chow2017discrete}. The \emph{gradient flow} of the Fisher information has been studied in \cite{gianazza2009wasserstein} as a model for the quantum drift diffusion. Using the hot gas experiment, we can give an heuristic for equation \eqref{eq:2ndorderODE} to hold. The marginal flow $(\mu_t)$ of SB models the empirical measure of a particle system which minimizes the relative entropy on path space while going from $\mu$ to $\nu$. Therefore, if there was no constraint on the final end point, $(\mu_t)$ would be the gradient flow of the \emph{marginal entropy} started at $\mu$
\be\label{eq:gflowent} v_t  = - \frac{\sigma}{2}\otgrad \ent (\mu_t), \quad \mu_0 = \mu, \ee
where $v_t$ is the velocity field of $\mu_t$ and $\scrH_{U}$ the relative entropy on $\cP_2(M)$:
\be\label{eq:ent}
\ent(\mu) = \begin{cases} \int_M (\log \mu + 2U) d\mu \quad & \mbox{if $\mu \ll \vol $}  \\  + \infty & \mbox{otherwise}  \end{cases}
\ee
If we differentiate once more this relation to be allowed to impose the terminal condition, we formally obtain 
\beas \frac{\bD}{dt} v_t &\stackrel{\eqref{eq:gflowent}}{=}& -\frac{\sigma}{2} \frac{\bD}{dt} \otgrad \ent(\mu_t) \\
&=&  -\frac{\sigma}{2}\hess_{\mu_t} \ent (v_t) \\ &\stackrel{\eqref{eq:gflowent}}{=}&  \frac{\sigma^2}{4} \hess_{\mu_t} \ent (\otgrad \ent )  \\ &=& \frac{\sigma^2}{8} \otgrad \fish (\mu_t),\eeas

where to get the last equation we used the fact that the Fisher information is the squared norm of the gradient of the entropy,
 \[\fish(\mu) = |\otgrad \ent(\mu)|^2_{L^2_{\mu}} .\]

Theorem \ref{thm:2ndorderODE} gives an answer to the problem of determining what second order equation should the bridge of a diffusion satisfy. Other authors (see e.g. \cite{Kre88,Kre97,CruZa91,Th93,Zam86}) have proven results in this direction. In this respect, equation \eqref{eq:2ndorderODE} has some nice features. The first one is that the acceleration we consider here is a ``true" acceleration, in the sense that it can be constructed as the covariant derivative associated with a Riemannian structure. Moreover, as we shall see in Theorem \ref{thm:entropybound}, using the Riemannian formalism we can obtain quantitative results for the dynamics of Schr\"odinger bridges. The following kind of conservation law is also immediately derived: it can be compared with similar results obtained in the above mentioned articles.
\begin{cor}\label{cor:encons}
	If either Theorem \ref{thm:2ndorderODE} or Theorem \ref{thm:2ndorderODEb} holds,  there exists a finite constant $c(\mu,\nu)$ such that
\be\label{eq:encons1}	
\forall \, 0 < t < 1, \quad \frac{1}{\sigma^2} |v_t|^2_{\tsp_{\mu_t}} - \frac{1}{4} \fish(\mu_t) =c(\mu,\nu),
\ee
where $|\cdot|_{\tsp_{\mu_t}} = |\cdot |_{L^2_{\mu_t}}$ is the norm taken in the tangent space at $\mu_t$ (see section 2.3 for details).
\end{cor}
\subsection{Quantitative results for Schr\"odinger bridges}
 Going back once more to the hot gas experiment is useful to give a qualitative description of the dynamics of SB. Indeed, the marginal flow $(\mu_t)$ is the empirical measure of a particle system that, at the same time
\begin{itemize}
\item has to minimize entropy: this means that particles are willing to arrange according to the equilibrium configuration for the Langevin dynamics, which we denote $\inv$ (recall that $d\inv = \exp(-2U)d\vol$).
\item has to reach an \textit{unexpected} final configuration $\nu$, which looks very different from $\inv$.
\end{itemize}
Thus, we expect the dynamics to be divided into two phases. In the first one, entropy minimization dominates and $\mu_t$ relaxes towards $\inv$. In the second phase, the necessity to attain the configuration $\nu$ at $t=1$ prevails: particles start arranging according to $\nu$ and $(\mu_t)$ drifts away from $\inv$.  From this sketch, it is clear that \nameref{secondquest} is the crucial question to address when studying the dynamics of Schr\"odinger bridges. Our answer is in the following Theorem, where we show that the Bakry \'Emery condition is equivalent to an upper bound for the relative entropy  of $\mu_t$ w.r.t. $\inv$, under the additional assumption that $M$ is compact. The proofs we present are also \emph{formally} correct in the non compact case, and we expect the compactness assumption not to be necessary. We make use of it to justify some integration by parts under $\inv$. 
\begin{theorem}\label{thm:entropybound}
Let $M$ be compact and $\scrL = \frac{1}{2} \Delta -\nabla U \cdot \nabla$. The following are equivalent:
\begin{enumerate}[(i)]
\item The Bakry \'Emery condition
\be\label{eq:Bakem}  \forall x \in M, \quad \ric_x +2 \mathbf{Hess}_x\, U \geq \lambda \, \id \ee  
\item For any $\mu,\nu$ such that $\ent(\mu),\ent(\nu)<+\infty$ and any $\sigma>0$, the marginal flow of $SB(\mu,\nu,\sigma \scrL)$ is such that
\bea\label{eq:entropybound}
\nonumber \forall t \in [0,1] \quad \entab(\mu_t)  \leq \frac{1-\exp(- \sigma\lambda(1-t))}{ 1 -\exp(-  \sigma \lambda )}\entab(\mu)+ \frac{1-\exp(- \sigma \lambda t)}{ 1 -\exp(- \sigma \lambda )}\ent(\nu) \\
-  \frac{\cosh(\frac{\sigma\lambda }{2})- \cosh(-\sigma\lambda(t-\frac{1}{2}))}{\sinh(\frac{\sigma\lambda }{2})}\entcost(\mu,\nu).
\eea	
\end{enumerate}
\end{theorem}

\begin{remark}
A simple case when our proof works beyond compactness is when Assumption \ref{ref:mainass} (B) holds, i.e. when $\bP$ is a stationary Ornstein Uhlenbeck process.
\end{remark}

In the article \cite{leonard2013convexity} the author proves a representation of the first and second derivative of the entropy along SBs using the operators $\Gamma$ and $\Gamma_2$. From this, he deduces that the entropy is convex under the condition \eqref{eq:Bakem}. However, no quantitative estimate such as \eqref{eq:entropybound} is proven there. Having equation \eqref{eq:2ndorderODE} at hand, we can give a geometric interpretation to the results of \cite{leonard2013convexity},  turn them into quantitative estimates and gives some answers to the questions raised there. 
We shall give a first ``geometric" proof of $(i) \Rightarrow (ii)$ at Theorem  \ref{thm:entropybound}, and then a second proof based on $\Gamma$-calculus, which follows the blueprint of the first proof. 
\begin{remark}
	Since the functions $\frac{1-\exp(-\sigma \lambda(1-t))}{ 1 -\exp(- \sigma \lambda t)}$ and $\frac{1-\exp(-\sigma \lambda t)}{ 1 -\exp(- \sigma \lambda t)}$ are increasing in $\lambda$ and only the term involving $\entcost(\mu,\nu)$ is decreasing, one might think that the bound \eqref{eq:entropybound} could get worse by increasing $\lambda$. This is false because of how $\entcost(\mu,\nu)$ depends on the marginal entropies. 
	It can be read off the proof of Theorem \ref{thm:entropybound} that for any $\mu,\nu$ fixed the right hand side of \eqref{eq:entropybound} is strictly decreasing in $\lambda$. 
\end{remark}
\begin{figure}[h]
\includegraphics[scale=0.5]{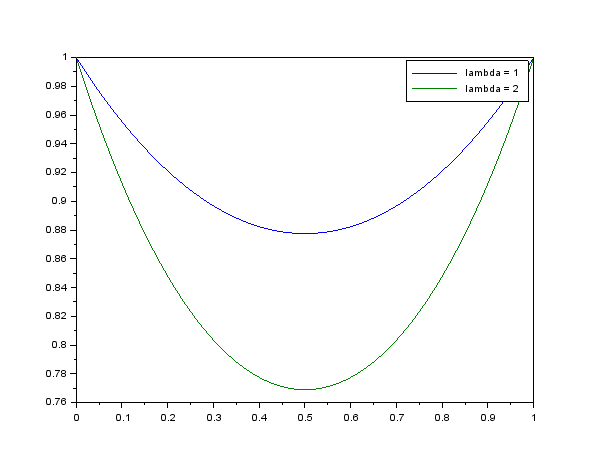}
\caption{If $M=\R, \scrL = \frac{1}{2} \Delta - x \cdot  \nabla $, then $\bP$ is a stationary OU process and the condition \eqref{eq:Bakem} holds with $\lambda =1$. In blue, we plot the corresponding bound \eqref{eq:entropybound} for $\sigma=1$. We chose $\mu = \nu$ to  be a Gaussian law with mean zero and variance $0.25$. The entropic cost turns out to be around 0.6841. In green the same bound for $ \scrL = \frac{1}{2} \Delta - 2x  \cdot \nabla $. In this case $\lambda =2$, and the bound is thus stronger.} 
\end{figure}

\subsection{A functional inequality for the entropic transportation cost}
Setting $t=\frac{1}{2}$ in \eqref{eq:entropybound}, we obtain
\begin{cor}\label{cor:enttrans1}
	Let $M$ be compact and \eqref{eq:Bakem} hold for some $\lambda>0$. Then, for any $\mu,\nu \in \cP(M)$ and $\sigma>0$
	\be\label{eq:sensitivitybound}
	\entcost(\mu,\nu) \leq \frac{1}{1-\exp(-\frac{\sigma \lambda}{2})} (\ent(\mu)+\ent(\nu)).
	\ee
\end{cor}

To interpret this inequality, recall that if $\hat{\bQ}$ is the Schr\"odinger bridge between $\mu$ and $\nu$ then
\[ \scrH(\hat{\bQ}| \bP) =\entcost(\mu,\nu). \]
The bound \eqref{eq:sensitivitybound}  tells that the entropic cost is at most linear in the marginal entropies with a constant that improves with curvature. Since $\entcost(\mu,\nu)$ is a relative entropy on path space, \eqref{eq:sensitivitybound} may as well be seen as a partial converse to the fact that joint entropies dominate marginal entropies. Such bound may be of particular interest in practice. Indeed, it is a very hard task to compute $\entcost(\mu,\nu)$. However,
 $\ent(\mu),\ent(\nu)$ can be computed directly from the data of the problem.
  With some simple computation we also derive from Theorem \ref{thm:entropybound} the following Corollary.
\begin{cor}\label{cor:enttrans2}
Let $M$ be compact and \eqref{eq:Bakem} hold for some $\lambda>0$. Then, for any $\mu \in \cP(M)$ and $\sigma>0$
\be\label{eq:enttransineq2}
	\ent (\mu) \leq \entcost(\mu,\inv) \leq \frac{1}{1-\exp(-\sigma\lambda)} \ent(\mu).
	\ee
\end{cor}
We shall see in the next section how \eqref{eq:enttransineq2} is connected with Talagrand's inequality. Let us remark here that the entropic transportation cost falls into the class of generalized costs considered in \cite{gozlan2017Kantorovich}. 
\subsection{Small noise limits and relations with classical optimal transport}\label{sec:smallnoise}
In this section we ``slow down"  Brownian motion, i.e. consider the generator
\[ \scrL^{\varepsilon} = \frac{\varepsilon}{2} \Delta, \]
assuming for simplicity that $U=0$. Because of this choice, $\ent$ is equal to the relative entropy against the volume measure, which we abbreviate with $\scrH$. Our aim is to provide an heuristic showing that the bounds obtained above are consistent with the classical results of optimal transport. We will make this argument rigorous in the proof of Theorem \ref{thm:entropybound}. In particular, we shall see how the convexity of the entropy along displacement interpolations is the small noise limit of the estimate \eqref{eq:entropybound}, and Talagrand's transportation inequality (see \cite{talagrand1996transportation}, \cite{otto2000generalization}) is the small noise limit of \eqref{eq:enttransineq2}. Let \eqref{eq:Bakem} hold and $(\mu^{\veps}_t)$ be the marginal flow of $\mathrm{SB}(\mu,\nu,\scrL^{\veps})$; the bound \ref{thm:entropybound} says that
\bea\label{eq:smallnoiseentropybound}
\nonumber \scrH(\mu^{\varepsilon}_t )  \leq \frac{1-\exp(- \veps \lambda (1-t))}{ 1 -\exp(-  \varepsilon \lambda)}\scrH(\mu )+ \frac{1-\exp(- \veps \lambda t)}{ 1 -\exp(-  \veps\lambda)}\scrH(\nu ) \\-  \frac{\cosh(\frac{\veps\lambda }{2})- \cosh(-\veps\lambda  (t-\frac{1}{2}))}{\sinh(\frac{\veps\lambda}{2})} \cT_{\scrH}^{\veps}(\mu,\nu),
\eea
 
Moreover, Corollary \ref{cor:enttrans2} is
\be\label{eq:smallnoise2} \forall \mu \in \cP(M),\quad  \cT^{\veps}(\mu,\inv) \leq \frac{1}{1-\exp(- \lambda \veps)} \scrH(\mu ). \ee
In the articles \cite{Mik04,leonard2012schrodinger,gigli2017second} connections between the Schr\"odinger problem and the Monge Kantorovich problem were established. In particular, in \cite[Th 3.6]{leonard2012schrodinger} a $\Gamma$-convergence result for the ``small noise" Schr\"odinger problem towards the Benamou-Brenier formulation of the Monge Kantorovich problem is proven. Roughly speaking, this implies that
\[ (\mu^{\varepsilon}_t )\rightarrow (\mu^{0}_t), \quad \varepsilon \cT^{\varepsilon}(\mu,\nu) \rightarrow \frac{1}{2} W^2_2(\mu,\nu), \]
 where $(\mu^{0}_t)$ is the displacement interpolation between $\mu$ and $\nu$.
Using these results, and letting $\varepsilon \rightarrow 0$ in \eqref{eq:smallnoiseentropybound}, one recovers the convexity of the entropy along displacement interpolations, whereas in \eqref{eq:smallnoise2} one recovers Talagrand's inequality. 

\subsection{Reciprocal characteristics and Fisher information}
\subsubsection*{Reciprocal characteristics}
In this section we take $M=\RD,\sigma =1$. $\bP$ is then the stationary Langevin dynamics for the generator
\[ \scrL = \frac{1}{2}\Delta - \nabla U \cdot \nabla \quad .  \]
The aim of this section is to point out some connections between convexity of the \emph{reciprocal characteristic} and convexity of the Fisher information along SBs. Let us recall the definition of \emph{reciprocal characteristic} associated with a smooth potential $U$: it is the vector field $\nabla \scrU$, where we recall that 
\[ \scrU(x) := \frac{1}{2} |\nabla U|^2(x)- \frac{1}{2} \Delta U(x).\]

In \cite{Kre88,Cl91} it is shown how the reciprocal characteristic is an invariant for the family of bridges of the Langevin dynamics, and, more generally, for the associated  \emph{reciprocal processes}. We refer to the survey \cite{LRZ} and the articles \cite{Kre97,RT02,RT05,CL15} for more information about reciprocal processes and reciprocal characteristics.
In Theorem \ref{thm:2ndorderODEb} the reciprocal characteristic appears as one of the terms expressing the acceleration of a Schr\"odinger bridge.
  In the recent work \cite{conforti2016gradient}, a number of quantitative results about the bridges of the Langevin dynamics were obtained, under the hypothesis that $\scrU$ is uniformly convex. In particular at Theorem 1.1, an equivalence between $\alpha^2$ convexity of $\scrU$, the possibility of constructing certain couplings for bridges with different endpoints, and a gradient estimate along the bridge semigroup is proven. We want to add one more motive of interest for the case when $\scrU $ is convex showing that it implies lower bounds for the second derivative of the modified Fisher information along SBs. 
\subsubsection*{Convexity of the Fisher information}
If $\scrU$ is convex, we are able to show that $t \mapsto \fishu(\mu_t)$ is convex, provided some technical assumptions hold and the optimal coupling $\hat{\pi}$ of \ref{eq:SPs} is a log-concave measure. In the next Theorem, since we are in the flat case, we prefer to write $D v \cdot w$ instead of $\nabla_w v$ for the derivatives of vector fields.
\begin{theorem}\label{thm:fishdecay}
Let $U,\mu,\nu$ be such that
\begin{enumerate}[(i)]
\item  The hypothesis of Theorem \ref{thm:2ndorderODEb} are satisfied.
\item The vector field \[ \partial_t \otgrad \fishu(\mu_t)+D \otgrad \fishu(\mu_t) \cdot v_t \] is bounded in $L^2_{\mu_t}$ uniformly in $t\in [0,1]$.
\item If we define $A_t,B_t$ by
\[ A_t(x) := \sup_{\substack{1 \leq k \leq 4 \\ i_1..i_k \in \RD}}  |\partial_{x_{i_k}}..\partial_{x_{i_1}} \log \mu_t(x) |, \quad B_t(x):= \sup_{\substack{1 \leq k \leq 4 \\ 1\leq j \leq d \\ i_1..i_k \in \RD}}  |\partial_{x_{i_k}}..\partial_{x_{i_1}} v^j_t(x) |,\]
then
\bes
\sup_{t \in [0,1]} \int_{\RD}  ( A_t B_t)^2( \max_{1\leq i \leq d}|\partial_i \mu_t|+\mu_t) d \vol <+\infty.
\ees 
\item The optimal coupling $\hat{\pi}$ in \ref{eq:SPs} is log-concave.
\item $\scrU$ is $\alpha^2$ convex.
\end{enumerate}
Then $t \mapsto \fishu(\mu_t)$ is convex. Moreover, we have the following bound:
\be\label{eq:fishsecder} \partial_{tt} \fishu(\mu_t) \geq  8 \alpha^2 |v_t|^2_{\tsp_{\mu_t}}+  \frac{1}{8}  |\otgrad \fishu(\mu_t) |^2_{\tsp_{\mu_t}}.\ee

\end{theorem}
A simple example to which the Theorem can be applied is when $\bP$ is a stationary Ornstein Uhlenbeck process and $\mu,\nu$ are Gaussian laws. Concerning the assumptions of the Theorem, we believe that (iii) could be largely weakened , and that the assumption that $\hat{\pi}$ is log concave cannot be significantly weakened. Indeed, the key argument in the proof of Theorem \ref{thm:fishdecay} is to identify the log concave measures as ``convexity points" for the Fisher information functional $\scrI$ (corresponding to the case $U=0$). Let us explain this. In Lemma \ref{lm:covdevfish} we prove that \footnote{We recall that $(v_t)$ is the velocity field of $(\mu_t)$, $\langle \cdot, \cdot \rangle_{\tsp_{\mu_t}}$ is the inner product in $L^2_{\mu_t}$ and $\frac{\bD}{dt}$ the covariant derivative. We also denote $Dv_t$ the Jacobian matrix of $v_t$. Finally, we abbreviate $\partial_{x_i}$ with $\partial_i$, and adopt the same convention for higher-order derivatives.}
\bea\label{eq:rec3} \nonumber \langle \frac{\bD}{dt} \otgrad \scrI (\mu_t) , v_t \rangle_{\tsp_{\mu_t}} &=&    2 \int | D v_t \cdot \nabla \log \mu_t  +  \nabla \mathbf{div}(v_t) |^2 d\mu_t  \\&-& 4 \int \sum_{k,j=1}^d (D v_t\cdot Dv_t)_{kj} \partial_{kj} \log \mu_t d\mu_t . \eea
 Schur's product theorem \cite[Th 7.5.3]{horn2012matrix} makes sure that if $ -\mathbf{Hess} \log \mu_t$ is positive definite then $\sum_{k,j=1}^d (D v_t\cdot D v_t)_{kj} \partial_{kj} \log \mu_t >0$. Since $\langle \frac{\bD}{dt} \otgrad \scrI (\mu_t) , v_t \rangle_{\tsp_{\mu_t}} $ morally stands for 
$\langle \hess_{\mu_t}(v_t),v_t\rangle_{\tsp_{\mu_t}} $, equation \eqref{eq:rec3} indicates that log-concave measures are convexity points for $\scrI$. The same formula suggests that the Fisher information is not a convex function in general. 
For the modified Fisher information $\fishu$ we get
\bea\label{eq:rec2} \nonumber \langle \frac{\bD}{dt} \otgrad \fishu (\mu_t) , v_t \rangle_{\tsp_{\mu_t}} &=&    2 \int | D v_t \cdot \nabla \log \mu_t  +  \nabla \mathbf{div}(v_t) |^2 d\mu_t  \\&-& 4 \int \sum_{k,j=1}^d (D v_t\cdot Dv_t)_{kj} \partial_{kj} \log \mu_t d\mu_t \\
&+& 8\int \langle v_t , \mathbf{Hess}\scrU\cdot v_t \rangle d\mu_t. \eea
 This formula clearly indicates that the convexity of $\scrU$ contributes to the convexity of $\scrI_U$. 
\begin{remark}
If we reintroduce the dependence on $\sigma$ in Theorem \ref{thm:fishdecay} the analog of \eqref{eq:fishsecder} is
\[  \partial_{tt} \fish(\mu_t) \geq  8 \alpha^2 |v_t|^2_{\tsp_{\mu_t}}+  \frac{\sigma^2}{8}  |\otgrad \fish(\mu_t) |^2_{\tsp_{\mu_t}}.\]
Letting $\sigma \rightarrow 0$ and using the convergence of the SB towards the displacement interpolation, it is natural to guess that the modified Fisher information is convex along a geodesic $(\mu_t)$ provided $\mu_t$ is log concave for all $t$. At the moment of writing, we have no rigorous proof of this.
\end{remark}

\subsection{Feynman-Kac penalisation of Brownian motion.}
In this section, the setting is $M=\RD$,\,$\sigma=1,U=0$. We look at a family of stochastic processes called \emph{Feynman-Kac penalisations} of Brownian motion, see  the monograph \cite{roynette2009penalising}. If $\bP$ is the stationary Brownian motionon $\RD$ (which is not a probability measure), a Feynman-Kac penalisation $\bQ$ \cite[Ch. 2]{roynette2009penalising} is a probability measure  on $\Omega$ which can be written as 
\[  d \bQ = \frac{1}{Z} f(X_0) \exp( - \int_{0}^1 K(X_s) ds ) d \bP, \]
where $K$ is a smooth lower bounded potential, $f$ a positive integrable function and $Z$ a normalization constant. As before, for any $t\in [0,1]$ we define the functions
\bea\label{def:fFKpen}
f_t(x) = E_{\bP}\big[ f(X_0)\exp( - \int_{0}^t K(X_s) ds) \big| X_t=x \big], \\
\nonumber g_t(x) = E_{\bP}\big[ \exp( - \int_{t}^1 K(X_s) ds)  \big| X_t=x \big].
\eea
We will also consider the energy functional $\scrE_K$ defined as
\be\label{eq:endef} \scrE_K(\mu) = \int_M K d \mu.\ee
\begin{theorem}\label{thm:FK}
Assume that the marginal flow $(\mu_t)$ of $\bQ$  is such that 
\be\label{eq:FK1} \forall \veps \in (0,1) \quad \sup_{t \in [\veps,1-\veps] } |\otgrad \scrI(\mu_t) + \scrE_K(\mu_t)|_{L^2(\mu_t)}<+\infty  \ee
and $f_t,g_t$ are such that for all $\veps \in (0,1)$
\be\label{eq:FK2}  \sup_{t \in [\veps,1-\veps]} |\nabla (\log g_t-\log f_t)|_{L^2_{\mu_t}}<+\infty ,\quad \sup_{\substack{t \in [\veps,1-\veps]\\ x \in M}} |\mathbf{Hess}_x (\log g_t - \log f_t) |_{op}<+\infty. \ee
Then the marginal flow $(\mu_t)$ is a regular curve and its velocity field $(v_t)$ is absolutely continuous. Moreover, $(\mu_t)$ satisfies the equation
\bes \frac{\bD}{dt}  v_t =\otgrad (\frac{1}{8} \scrI + \scrE_K) 	(\mu_t),\ees
where $\frac{\bD}{dt} v_t$ is the \emph{covariant derivative} of $(v_t)$ along $(\mu_t)$.   
\end{theorem}

\section{The Riemannian structure of optimal transport}\label{sec:2}
\subsection{Preliminaries and notation} We consider a smooth Riemannian Manifold $(M,g)$ which is complete, connected, closed, without boundary, and of finite dimension. We denote the Levi-Civita connection of $(M,g)$ by $\levciv$. For any $\veps \in (0,1) $, we define $D_{\veps}:=[\veps,1-\veps]\times M$ and let $\cinf$ be the space of smooth functions on $M$. $\cinfp,\cinfc \subseteq \cinf$ are the subset of positive function and the subset of bounded functions whose derivatives are all bounded.  We also set $\cinfcp := \cinfc \cap \cinfp$. Finally $\cC^{c}_{\infty}\subseteq \cinf $ is the subset of compactly supported functions and $\cC^{c,+}_{\infty}:= \cC^{c}_{\infty} \cap \cC^{+}_{\infty}$.
The definition of $\cinf,\cinfc,\cinfcs$ naturally extends to vector fields over $M$, to functions whose domain is $[0,1] \times M$ instead of $M$, and to families of vector fields over $M$ indexed by time.
We call $\cP_2(M)$ the space of probability measures on $M$ admitting second moment. The Wasserstein distance $W_2(\mu,\nu)$ between $\mu,\nu\in \cP_2(M)$ is given by
 
\[ W_2^2(\mu,\nu) := \inf_{\pi \in \Pi(\mu,\nu)} \int_{M^2} d_M(x,y) \pi(dxdy), \]
where $d_M$ is the Riemannian distance on $M$.
In the rest of the paper, curves are always defined on the time interval $[0,1]$, unless otherwise stated.
In this section, we shall describe a kind of Riemannian structure on $\cP_2(M)$ associated with $W_2(\cdot,\cdot)$. We do not prove new results and simply extract definitons and results from \cite{ambrosio2013user} and \cite{giglisecond}, to which we refer for the proofs.

\subsection{Geodesics, Velocity fields}

Recall that if $(x_t)$ is a curve in a generic metric space $(\cX,d)$, we say that it is \emph{absolutely continuous} over $[\veps,1-\veps]$ provided that for some integrable function $f$ 
\[\forall \veps \leq r<s \leq 1-\veps \quad d(x_r,x_s) \leq \int_{r}^s f(t)dt. \]
In all what follows, we will write "absolutely continuous curve" and mean "absolutely continuous over $[\veps,1-\veps]$ for all $\veps \in (0,1) $ ". A curve is a \emph{constant speed geodesic} if and only if
\[\forall s,t \in [0,1] \quad  d(x_s,x_t) = |t-s| d(x_0,x_1). \]
$(\cX,d) $ is said to be a geodesic space provided that for any pair of points there exist a constant speed geodesic connecting them.
It turns out  that (\cite[Th 2.10]{ambrosio2013user}) $(\cP_2(M), W_2(\cdot,\cdot))$ is a geodesic space.
If $(\mu_t)$ is an absolutely continuous curve on $(\cP_2(M),W_2(\cdot,\cdot))$ then one can show (\cite[Th. 2.29]{ambrosio2013user}) that there exists a Borel family of vector fields $(v_t)$ such that the \emph{continuity equation} 
\be\label{eq:coneq} \partial_t \mu_t + \nabla \cdot (\mu_t v_t) =0 \ee
holds in the sense of distributions. Moreover, it can be shown that there exists a unique up to a negligible set of times family of vector fields $(v_t)$ satisfying \eqref{eq:coneq} and such that
\[    v_t \in \overline{\{ \nabla \varphi, \varphi \in \cC^c_{\infty} \}}^{L^2_{\mu_t}} \quad t- \text{a.e. \ .} \]
We call this the \emph{velocity field} of $(\mu_t)$. Conversely, if $(v_t)$ is a Borel family of vector fields satisfying \eqref{eq:coneq} such that  $\forall \veps \in (0,1)$  $\int_{\veps}^{1-\veps}|v_t|_{L^2_{\mu_t}}dt <+\infty$ and $v_t \in \overline{\{ \nabla \varphi, \varphi \in \cC^c_{\infty} \, \}}^{L^2_{\mu_t}} t- \text{a.e.}$, then $(\mu_t)$ is absolutely continuous and $v_t$ is its velocity field.
\subsection{Tangent space and Riemannian metric}
The tangent space at $\mu \in \cP_2(M)$ is generated by the ``space of directions"
\bes \mathit{Geod}_{\mu} = \{ \text{const. speed geodesics starting from $\mu$ defined on some interval $[0,T]$} \} /\approx \ees

where $(\mu_t) \approx (\mu'_t)$ provided they coincide on some right neighborhood of $0$. We equip  $\mathit{Geod}_{\mu}$ with the distance
\[ D((\mu_t),(\mu'_t)) = \lim_{t \downarrow 0} \frac{1}{t}W_2(\mu_t,\mu'_t). \]
The \emph{Tangent space} $\tsp_{\mu}$ is defined as the completion of $\mathit{Geod}_{\mu}$ w.r.t. $D$.
The tangent space can be nicely described using the completion in $L^2_{\mu}$ of the ``space of gradients "
by considering the map $ \iota_{\mu}$
\[\iota_{\mu} : \Big(\overline{\{ \nabla \varphi, \varphi \in \cC^c_{\infty} \}}^{L^2_{\mu}}, d_{L^2_{\mu}}\Big) \rightarrow (\tsp_{\mu},D), \quad \iota_{\mu} \nabla \varphi  
\mapsto (\mu^{\varphi}_t),\]
where $(\mu^{\varphi}_t)$ is the unique (modulo $ \approx$) constant speed geodesic originating from $\mu$ and such that $v_0=\nabla \varphi$. If $\mu$ is a regular measure in the sense of \cite[Def 1.25]{ambrosio2013user}, then $\iota_{\mu}$ can be extended to a bijective isometry. Therefore, $\tsp_{\mu}$ inherits the inner product $\langle .,.\rangle_{\tsp_{\mu}}$ from that of $L^{2}_{\mu}$
\[ \langle \iota_{\mu} \nabla \varphi,\iota_{\mu}\nabla \psi \rangle_{\tsp_{\mu}} := \langle \nabla \varphi,\nabla \psi\rangle_{L^2_{\mu}} . \]

For this reason, in what follows we make no distinction between the elements of the two different spaces, and write 
$\langle \nabla \varphi,\nabla \psi \rangle_{\tsp_{\mu}}$ instead of $\langle \iota_{\mu} \nabla \varphi , \iota_{\mu} \nabla \psi \rangle_{\tsp_{\mu}}$. Similarly, we write $|\nabla \varphi|_{\tsp_{\mu}}$ instead of $ |\iota_{\mu} \nabla \varphi|_{\tsp_{\mu}} $.

The Benamou-Brenier formula\footnote{In the original result of Benamou and Brenier $v_t$ is not the velocity field of $(\mu_t)$, but just an arbitrary weak solution to the continuity equation. However, it is easy to see that the representation formula for the Wasserstein distance remains true if we restrict the minimization to the couples $(\mu_t,v_t)$ such that $(v_t)$ is the velocity field of $(\mu_t)$.} \cite{benamou2000computational} shows how the metric we have just introduced  is morally the Riemannian metric for $(\cP_2(M),W_2(\cdot,\cdot))$. Indeed it says that
\bes
W^2_2(\mu,\nu) = \inf_{\mu_t,v_t} \int_{0}^1| v_t|^2_{\tsp_{\mu_t}} dt,
\ees
where $(\mu_t)$ varies in the set of absolutely continuous curves joining $\mu$ and $\nu$ and $(v_t)$ is the velocity field of $(\mu_t)$.
\subsection{Regular curves and flow maps}
We give the definition of regular curve following closely \cite[Def. 2.8]{giglisecond}, the only difference being that we define regularity over $[\veps,1-\veps]$ for $\veps \in (0,1)$, instead of looking at $[0,1]$. 
\begin{mydef}\label{def:reg}
For $\veps \in (0,1)$, an absolutely continuous curve $(\mu_t)$ is \emph{regular} over $[\veps,1-\veps]$ provided
\be\label{eq:myreg1} \int_{\veps}^{1-\veps} |v_t|^2_{\tsp_{\mu_t}} dt < + \infty \ee
and 
\be\label{eq:myreg2} \int_{\veps}^{1-\veps} \text{L} (v_t) dt < + \infty, \ee
where for a smooth vector field $\xi$, $\text{L}(\xi)$ is defined as 
\[  \text{L}(\xi) = \sup_{\substack{x\in M \\ w:|w|=1}} |\levciv_w \xi(x)|  .\]
\end{mydef}
For non smooth vector fields, the general definition of $\text{L}$  can be found at \cite[Def 2.1]{giglisecond};  in this article we will only be concerned with the smooth case. In all what follows, by regular curve we mean ``\emph{regular over } $[\veps, 1-\veps]$ \emph{for all} $\veps \in (0,1)$". 
If $(\mu_t)$ is a regular curve and $(v_t)$ its velocity field, there exists a unique family of maps, called \emph{flow maps}, $\bT(t,s,\cdot): \supp \mu_t \rightarrow \supp \mu_s $ such that for any $\veps \in (0,1)$, $t \in (0,1)$, $x \in \supp \mu_t$  the curve $s \mapsto \bT(t,s,x)$ is absolutely continuous over $(\veps,1-\veps)$  and satisfies
\be\label{eq:flowmaps} 
\begin{cases}
\frac{d}{ds} \bT(t,s,x) = v_s(\bT(t,s,x)) \quad a.e. \ s \in (\veps, 1-\veps), \\
\bT(t,t,x)= x.
\end{cases}
 \ee
 The flow maps enjoy the following properties:
\be\label{eq:flowppty} \bT(r,s,\cdot) \circ \bT(t,r,\cdot) = \bT(t,s,\cdot) , \quad \bT(t,s,\cdot)_{\#} \mu_t = \mu_s .\ee

\subsection{The maps $ (\tau_x)^t_s $ and $\tau^t_s(u) $}
These maps are needed to define the covariant derivative. The following are Definition 2.9 and 2.12 in \cite{giglisecond}.
\begin{mydef}\label{def:trmap1}
Let $(\mu_t)$ be a regular curve and $\bT(t,s,\cdot)$ its flow maps. Given	$t,s \in [0,1]$ and $x\in \supp \mu_t$, we let 
$(\tau_x)^t_s: T_{\bT(t,s,x)}M \rightarrow T_xM $ be the map which associate to $v \in T_{\bT(t,s,x)}M$ its parallel transport along the absolutely continuous curve $r \mapsto \bT(t,r,x)$ from $r=s$ to $r=t$.
\end{mydef}
The map $(\tau_x)^t_s$ is used to translate vectors along regular curves in $M$.  In the next definition we shall see how the maps $\tau^t_s$ do the same for vector fields along regular curves in $\cP_2(M)$. It should be stressed that the next definition is \emph{not} the parallel transport. Indeed, in general, $\tau^t_s(u)$ may not be in $\tsp_{\mu_t}$. 

\begin{mydef}\label{def:trmap2}
	Let $(\mu_t)$ be a regular curve, $\bT(t,s,\cdot)$ its flow maps and $u$ a vector field in $L^2_{\mu_s}$.
Then $\tau^t_s(u)$ is the vector field in $L^2_{\mu_t}$
defined by
\[ \tau^t_s(u)(x) = (\tau_x)^t_s (u \circ \bT(t,s,x)).  \]
\end{mydef}
The maps $\tau^t_s$ have similar properties to the flow maps: they enjoy the group property 
\be\label{eq:grouppty} \tau^t_s =\tau^t_r \circ \tau^r_s.
\ee
 Moreover $\tau^t_s$ is an isometry from $L^2_{\mu_s}$ to $L^2_{\mu_t}$, i.e.
\be\label{eq:isometry} 
\forall u \in L^2_{\mu_s}, \quad \int_{M} |u|^2 d \mu_s = \int_M  |\tau^t_s(u)|^2  d \mu_t,
\ee
or equivalently
\bes
|u|_{\tsp_{\mu_s}} = |\tau^t_s(u)|_{\tsp_{\mu_t}}.
\ees
\subsection{Vector fields along a curve}
\begin{mydef}
A vector field along a curve $(\mu_t)$ is a Borel map $(t,x) \mapsto u_t(x) \in \bT_xM$ such that $u_t \in L^2_{\mu_t}$ for a.e. $t$. It will be denoted by $(u_t)$. 
\end{mydef}
Observe that also non tangent vector fields are considered in this definition, i.e. $u_t$ may not be a gradient. Here is the definition of absolutely continuous vector field along a curve, see \cite[Def. 3.2]{giglisecond}.
\begin{mydef}\label{def:smooth}
Let $(u_t)$ be a vector field along the regular curve $(\mu_t)$ and $\tau^t_s(u)$ be given by Definition \ref{def:trmap2} . We say that $(u_t)$ is absolutely continuous over $[\veps,1-\veps]$ provided the map
\[ t \mapsto \tau^{t_0}_t(u_t) \in L^2_{\mu_{t_0}} \] is absolutely continuous in $L^2_{\mu_{t_0}}$ for all $t_0 \in [\veps, 1-\veps]$.
\end{mydef}
It can be seen that choice of $t_0$ is irrelevant. One could then set $t_0=\veps$ in the definition above. As before, by absolutely continuous vector field we mean ``\emph{absolutely continuous over} $[\veps,1-\veps]$ \emph{for all} $\veps \in (0,1)$".
\subsection{Total derivative and covariant derivative}
We are ready to define the total derivative of an absolutely continuous vector field  as in \cite[Def 3.6]{giglisecond}. Note that this is not yet the covariant derivative.
\begin{mydef}\label{def:totder} Let $(u_t)$ be an absolutely continuous vector field along the regular curve $(\mu_t)$. Its total derivative is defined as
\be\label{eq:totder} \frac{\bd}{dt}u_t := \lim_{h \rightarrow 0} \frac{\tau^t_{t+h}(u_{t+h})- u_t}{h} \quad  t-\text{a.e.},\ee
where the limit is intended in $L^2_{\mu_t}$.
\end{mydef}
To define the covariant derivative, we consider the orthogonal projection $P_{\mu}: L^2_{\mu} \rightarrow \overline{\{ \nabla \varphi, \varphi \in \cC^{c}_{\infty} \}}^{L^2_{\mu}}$. The following is \cite[Def. \, 6.8]{ambrosio2013user} for the flat case. For the general case, we refer to Definition 5.1 and discussion thereafter in  \cite{giglisecond}.

\begin{mydef}\label{def:covder}
	Let $(\mu_t)$ be an absolutely continuous and \emph{tangent} vector field along the regular curve $(\mu_t)$. Its covariant derivative is defined as
\bes
\frac{\bD}{dt} u_t := P_{\mu_t} ( \frac{\bd}{dt} u_t ) \quad t-\text{a.e.}.	
\ees
\end{mydef}
\subsection{Levi-Civita connection}
The covariant derivative just defined is the Levi-Civita connection, in the sense that it satisfies the \emph{compatibility of the metric} and \emph{torsion free identity}. The compatibility of the metric says that if $(u^1_t),(u^2_t)$ are tangent absolutely continuous vector fields along the regular curve $(\mu_t)$, then
\be\label{eq:metcomp} \partial_t \langle u^1_t,u^2_t \rangle_{\tsp_{\mu_t}} = \langle \frac{\bD}{dt} u^1_t,u^2_t \rangle_{\tsp_{\mu_t}}+ \langle  u^1_t,\frac{\bD}{dt}u^2_t \rangle_{\tsp_{\mu_t}}.\ee
For the torsion free identity, we refer to \cite[Eq. (6.10)]{ambrosio2013user} and \cite[Sec. 5.1]{giglisecond}.
\subsection{Gradient of a function}
Here, we do not seek for a very general definition, but rather give a customary one, well-adapted to our scopes, following \cite[Eq 3.50]{ambrosio2013user}. In this definition and in the rest of the paper by $\mu \in  \cC^+_{\infty}$ we mean that $\mu \ll \vol$ and $\frac{d\mu}{d\vol} \in \cC^+_{\infty}$.
\begin{mydef}\label{def:otgrad}
Let $\mu \in  \cC^+_{\infty}$. We say that $\scrF: \cP_2(M) \rightarrow \R \cup \{ \pm \infty \} $ is differentiable at $\mu$ if there exists a vector field $w \in \tsp_{\mu}$ such that for all $\varphi \in C^{c}_{\infty}$
\be\label{eq:graddef} \lim_{h \rightarrow 0} \frac{\scrF(\mu^{\varphi}_h  ) - \scrF(\mu) }{h} = \langle  w , \nabla \varphi \rangle_{\tsp_{\mu}}, \ee 
where $\mu^{\varphi}_h$ is the constant speed geodesic (modulo $\approx $ ) such that $\mu_0=\mu$ and $v_0=\nabla \varphi$.
\end{mydef}
It follows from our definition that if $\scrF$ is differentiable at $\mu$, then there exists a unique $w$ fulfilling \eqref{eq:graddef}. We then call $w$ the \emph{gradient} of $\scrF$ at $\mu$, and denote it $\otgrad \scrF(\mu)$.
Let us consider some functionals of interest. The gradient of the relative entropy $\ent$ is known to be
\be\label{eq:entgrad} \otgrad\ent(\mu) = \nabla \log \mu + 2\nabla U ,\ee
provided $\nabla \log \mu + 2 \nabla U \in L^2_{\mu}$ and $\mu$ is regular enough.
Thus, using \eqref{eq:fishdef} we obtain that
\be\label{eq:fishent} \fish(\mu) = |\otgrad \ent|^2_{\tsp_{\mu}}. \ee
The gradient of the modified Fisher information at $\mu$ is the vector field 
	\be\label{eq:fishgrad2} \otgrad \fish (\mu) = \nabla \Big( - |\nabla \log \mu |^2 - 2 \Delta \log \mu   + 8 \scrU \Big),  \ee
	provided the right hand side is in $L^2_{\mu}$ and $\mu$ regular enough (recall the definition of $\scrU$ given at \eqref{eq:fishgrad}).
All these computations can be found in \cite{von2012optimal}. We shall derive the expression \eqref{eq:fishgrad2} in the Appendix.
\subsection{Convexity of the entropy}
Recall that on a smooth finite dimensional Riemannian manifold whose Levi-Civita connection is $\overline{\nabla}$, the Hessian of $f$ at $x$ applied to $v \in T_{x}M $ is defined through (see e.g. \cite[Ex. 11 , pg. 141]{do1992riemannian})
\[ \mathbf{Hess}_xf(v) = \overline{\nabla}_{v} \nabla f (x) = \frac{\bD}{dt} \nabla f(x_t) |_{t=0}, \]
where $(x_t)$ is any curve such that $x_0=x$, $\dot{x}_0=v$. In this article we are only interested in defining a kind of Hessian for the entropy functional $\ent$. Therefore, as we did before, to simplify the definitions, we restrict to a very special setting.

\begin{mydef}\label{def:hessent}
Let $M$ be compact. Consider a measure $\mu \in \cC^{+}_{\infty}$ and a vector field $v \in \cC^b_{\infty} \cap \tsp_{\mu}$. We define
\[ \hess_{\mu} \ent (v) = \frac{\bD}{dt} \otgrad \ent (\mu_t)\Big|_{t=0}, \]
where $(\mu_t)$ is \emph{any} regular curve such that $\mu_0=\mu$, $v_0=v$ and $\mu_t \in \cC^{+}_{\infty}([0,\veps] \times M)$ for some $\veps >0$.
\end{mydef}
  It is easy to see that this is a good definition in the sense that there is always one curve fulfilling the requirements
and the value of the Hessian does not depend on the specific choice of the curve. The well known convexity of the entropy \cite{von2005transport} implies, in particular, that under the assumption \eqref{eq:Bakem} we have
\be\label{eq:entconv}
\forall \mu \in \cinfp,  v \in \cC^b_{\infty} \cap \tsp_{\mu}, \quad \langle \hess_{\mu} \ent (v),v \rangle_{\tsp_{\mu}} \geq \lambda  |v|^2_{\tsp_{\mu}}.
\ee
We shall give in the Appendix a formal proof of this fact, following \cite{otto2000generalization}.
\section{Proofs of the results}\label{sec:3}
\subsection{Proof of Theorems \ref{thm:2ndorderODE} and \ref{thm:2ndorderODEb} }
 This section is structured as follows: we first prove some preparatory Lemmas (Lemma \ref{lm:speedbound}, \ref{lm:speed}, \ref{covident}, \ref{lm:regcurv}). To prove these Lemmas, we rely on the technical Lemmas \ref{lm:logest} and \ref{lm:compactcase}, which we put in the Appendix. Finally, we prove the two Theorems. In the proofs of the Lemmas and of the Theorems, we assume for simplicity that $\sigma=1$. There is no difficulty in extending the proof to the general case. In the Lemmas \ref{lm:logest} and \ref{lm:compactcase} it is proven that if Assumption \ref{ref:mainass} holds, then the dual representation \eqref{eq:fgdec} holds as well. Thus, we will take it for granted in the next Lemmas and in the proof of the Theorem.
  In what follows, it is sometimes useful to use the functions $\tf_t$ and $\tg_t$ defined by (recall the definition of $f_t,g_t$ at \eqref{def:fg})
\be\label{def:tftg}
\tf_t(x) :=  \exp(-U(x)) f_t(x), \quad \tg_t(x) := \exp(-U(x)) g_t(x),
\ee
 The first Lemma is useful to identify the velocity field of $(\mu_t)$.
\begin{lemma}\label{lm:speedbound}
 Let Assumption \ref{ref:mainass} hold. If $f,g$ are given by \eqref{eq:fgdec} and $f_t,g_t$ by \eqref{def:fg}, then
\[\forall \veps \in (0,1) \quad  \sup_{t \in [\varepsilon,1-\varepsilon] } \int_{M} |\nabla \log g_t -\nabla \log f_t |^2 f_t g_t d \inv <+\infty. \] 
\end{lemma}
\begin{proof}
Assume that Assumption \ref{ref:mainass} (B)  holds. Then using point (ii) of Lemma \ref{lm:logest} we see that there exist constants $A,B$ such that
\[ \forall x \in \RD \sup_{t \in [\varepsilon,1-\varepsilon]}|\nabla \log g_t -\nabla \log f_t |^2 \leq A+B|x|^2.\]
The conclusion then follows from the boundedness of $f_t,g_t$ (see (i) in Lemma \ref{lm:logest}) and the fact that $\inv$ has Gaussian tails. If (A) holds, the conclusion follows directly from Lemma \ref{lm:compactcase}.\end{proof}
The next result is a representation for $(\mu_t)$ and its velocity field. It is a minor modification of the analogous results \cite[Sec. 5]{gentil2015analogy} and \cite[Sec. VI]{chen2014relation}, which build on the notion of \emph{current velocity} of a Markov diffusion process,  as introduced by Nelson in \cite{Nel67}. Thus, we postpone its proof the Appendix.
\begin{lemma}\label{lm:speed}
Let $f,g$ be given by \eqref{eq:fgdec} and $f_t,g_t,\tf_t,\tg_t$ by \eqref{def:fg},\eqref{def:tftg}. Then 
\be\label{eq:fgdect} \forall t \in [0,1], \quad d \mu_t = f_t g_t \, d\inv = \tf_t \tg_t d \vol  . \ee
Moreover, if the conclusion of Lemma \ref{lm:speedbound} holds,  $(\mu_t)$ is an absolutely continuous curve and its velocity field is 
\be\label{eq:velfield} (t,x) \mapsto \frac{\sigma}{2} \nabla(\log g_t - \log f_t) .\ee
\end{lemma}
The following Lemma gives a sufficient condition for the existence of the covariant derivative and its explicit form. It is a slight rearrangement of Example 6.7 in \cite{ambrosio2013user} and Equation (3.3) in \cite{giglisecond}, the only difference being that we do not assume that $\xi_t$ has compact support. Its proof is in the Appendix. 
\begin{lemma}\label{covident}
	Let $(\mu_t)$ be a regular curve, $(v_t)$ its velocity field and $(\xi_t)$ a $\cC_{\infty}$ vector field along $(\mu_t)$. If
	\be\label{eq:covident1} \forall \veps \in (0,1) \quad \sup_{ t \in [\varepsilon,1-\varepsilon]} |\partial_t \xi_t + \levciv_{v_t} \xi_t|_{\tsp_{\mu_t}}<+\infty,   \ee
	then $(\xi_t)$ is absolutely continuous along $(\mu_t)$ in the sense of Definition \ref{def:smooth} and
	\be\label{eq:covident2} \forall t \in (0,1), \quad \quad  \frac{\bd}{dt} \xi_t = \partial_t \xi_t + \levciv_{v_t} \xi_t .  \ee
	If $\xi_t=v_t$ then the we also have 
	\be\label{eq:covident3} \forall t \in (0,1), \quad \quad \frac{\bD}{dt} v_t = \frac{\bd}{dt} v_t = \partial_t v_t + \frac{1}{2} \nabla |v_t|^2 . \ee
\end{lemma}
In this Lemma, we establish that $(\mu_t)$ has the desired regularity properties.
\begin{lemma}\label{lm:regcurv}
Let Assumption \ref{ref:mainass} hold and $(\mu_t)$ be the marginal flow of \nameref{schrbr}. Then we have
\begin{enumerate}[(i)]
\item $(\mu_t)$ is a regular curve.
\item The velocity field $(v_t)$ is such that
\be\label{eq:regcurv3} \forall \veps \in (0,1) \quad  \sup_{t \in [\varepsilon,1-\varepsilon]}  \big|  \partial_t v_t + \frac{1}{2} \nabla |v_t|^2 \big|_{\tsp_{\mu_t}}<+\infty . \ee
\end{enumerate}
\end{lemma}
\begin{proof}
First, observe that, thanks to Lemma \ref{lm:speed}, we know that $v_t = \frac{1}{2}\nabla (\log g_t - \log f_t) $ and that $d \mu_t = f_t g_t d \inv$. Using this, if Assumption \ref{ref:mainass} (A) holds, the conclusion follows from Lemma \ref{lm:compactcase}. Therefore, let us assume that (B) holds. To prove that $(\mu_t)$ is regular we have to show \eqref{eq:myreg1} and \eqref{eq:myreg2}.  Therefore
\[\forall t \in (0,1), \quad | v_t |^2_{\tsp_{\mu_t}}  =\frac{\sigma^2}{4} \int_{M} |\nabla \log g_t -\nabla \log f_t |^2 f_t g_t d \inv \]
Using this identity, \eqref{eq:myreg1} follows from Lemma \ref{lm:speedbound}. To prove \eqref{eq:myreg2} observe that, since $v_t$ is of class $\cC_{\infty}$, we have that \[\text{L}(v_t) = \sup_{\substack{x\in \RD \\ w:|w|=1}} |\levciv_w v_t(x)|. \]
Using \eqref{eq:velfield}, we have for any $\veps \in (0,1)$:
\bea\label{eq:regcurv1} \sup_{ \substack{ t\in [\varepsilon,1-\varepsilon] \\ x \in \RD}} \sup_{w:|w|=1} |\nabla_w v_t(x)| \leq \frac{1}{2}  \sup_{ \substack{ t\in [\varepsilon,1-\varepsilon] \\ x \in \RD}} \sup_{\substack{w_2,w_1 \\ |w_2|,|w_1| \leq 1 }}|\partial_{w_2} \partial_{w_1} \log f_t | \\
\nonumber
+  \frac{1}{2}\sup_{ \substack{ t\in [\varepsilon,1-\varepsilon] \\ x \in \RD}} \sup_{\substack{w_2,w_1 \\ |w_2|,|w_1| \leq 1 }}|\partial_{w_2} \partial_{w_1} \log g_t |.\eea
Point (iii) of Lemma \ref{lm:logest} then immediately yields
\[ \sup_{t \in [\varepsilon,1-\varepsilon] } \text{L}(v_t) <+ \infty, \]
which proves \eqref{eq:myreg2} and that  $(\mu_t)$ is a regular curve. Let us now turn to the proof of \eqref{eq:regcurv3}. First, we observe that $\log f_t,\log g_t $ are classical solutions on $D_{\varepsilon} $ of the HJB equation
\be\label{eq:HJB}
\partial_t \log f_t = \scrL \log f_t + \frac{1}{2} |\nabla \log f_t|^2, \quad \partial_t \log f_t = -\scrL \log g_t- \frac{1}{2} |\nabla \log g_t|^2.
\ee
 Then, combining \eqref{eq:velfield} with point (ii) of Lemma \ref{lm:logest} and the HJB equation, we obtain that there exist constants $A,B$ such that uniformly on $D_{\veps}$
 \[ \Big| \partial_t v_t + \frac{1}{2} \nabla |v_t |^2 \Big|^2(x) \leq A+ B|x|^6. \]
 Moreover, point (i) in Lemma \ref{lm:logest} makes sure that $\sup_{D_{\veps}}f_t g_t\leq C$ for some $C<+\infty$. Thus,
\beas \sup_{t \in [\veps,1-\veps]} \int_M \Big| \partial_t v_t + \frac{1}{2} \nabla |v_t |^2 \Big|^2 d \mu_t &\stackrel{\eqref{eq:fgdect}}{=}& \sup_{t \in [\veps,1-\veps]}\int_M \Big| \partial_t v_t + \frac{1}{2} \nabla |v_t |^2 \Big|^2 f_t g_t  d \inv \\
&\leq & C \sup_{t \in [\veps,1-\veps]}\int_M A+B|x|^6   d \inv <+\infty,
 \eeas
 where to obtain the last inequality we used that $\inv$ has Gaussian tails. The desired conclusion follows. 
 \end{proof}

Let us now prove Theorem \ref{thm:2ndorderODE}.

\begin{proof}
Let $\varepsilon \in (0,1)$ and Assumption \ref{ref:mainass} hold. This entitles us to apply Lemma \ref{lm:speedbound} and \ref{lm:speed} to conclude that the velocity field of $(\mu_t)$ is given by \eqref{eq:velfield}. Lemma \ref{lm:regcurv} makes sure that $(\mu_t)$ is regular and that \eqref{eq:covident1} holds for $\xi_t=v_t$. Therefore we can apply Lemma \ref{covident} to conclude that $(v_t)$ is absolutely continuous and its covariant derivative is 
\[\frac{\bD}{dt} v_t = \partial_t v_t + \frac{1}{2} \nabla |v_t|^2 . \]
Next, we shall prove that 
\[ \partial_t v_t + \frac{1}{2} \nabla |v_t|^2 = \otgrad \scrI_U (\mu_t),\]
which concludes the proof. To do this, recall that from Lemma \ref{lm:speed}
\[ v_t = \frac{1}{2}\nabla( \log g_t - \log f_t) = \frac{1}{2}\nabla( \log \tg_t - \log \tf_t),\]
and that $\tf_t,\tg_t$ are classical solutions of \eqref{eq:FKpde} over $D_{\veps}$. Therefore, by taking $\log$ and using the positivity of $\tf_t,\tg_t$ we get 
\be\label{eq:HJBt} \partial_t \log \tf_t = \frac{1}{2}\Delta \log \tf_t + \frac{1}{2}|\nabla \log \tilde{f}_t|^2 -\scrU, \quad  \partial_t \log \tg_t=- \frac{1}{2}\Delta \log \tg_t -\frac{1}{2}|\nabla \log \tilde{g}_t|^2 +\scrU, \ee
where $\scrU$ is defined at \eqref{eq:fishgrad}. Combining these facts, we have
\beas
\partial_t v_t + \frac{1}{2}\nabla |v_t|^2 &=& -\frac{1}{2} \nabla \partial_t \log \tilde{f}_t +\frac{1}{2} \nabla \partial_t \log \tg_t + \frac{1}{2} \nabla \Big|\frac{1}{2} \nabla \log \tilde{g}_t - \frac{1}{2} \nabla \log \tilde{f}_t \Big|^2\\
&\stackrel{\eqref{eq:HJBt}}{=}& -\frac{1}{2} \nabla \big( \frac{1}{2}\Delta \log \tilde{f}_t + \frac{1}{2} |\nabla \log \tilde{f}_t|^2 -\scrU\big) \\
&+& \frac{1}{2} \nabla \big( - \frac{1}{2}\Delta \log \tilde{g}_t - \frac{1}{2} |\nabla \log \tilde{g}_t|^2 +\scrU \big)  \\
&+& \frac{1}{2} \nabla \big( \frac{1}{4}|\nabla \log \tilde{f}_t|^2 - \frac{1}{2} \langle \nabla \log \tilde{f}_t , \nabla \log \tilde{g}_t \rangle + \frac{1}{4} |\nabla \log \tilde{g}_t |^2  \big)\\
&=& -\frac{1}{4}\nabla (\Delta \log \tilde{f}_t+ \Delta\log \tilde{g}_t) -\frac{1}{8}\nabla \big ( |\nabla (\log \tilde{f}_t + \log \tilde{g}_t )|^2\big)+ \nabla \scrU \\
&\stackrel{\eqref{def:tftg}}{=}& -\frac{1}{4}\nabla \Delta \log \mu_t -\frac{1}{8}\nabla  |\nabla \log \mu_t |^2+ \nabla \scrU  \\
&\stackrel{\eqref{eq:fishgrad}}{=}& \frac{1}{8} \otgrad\scrI_U  (\mu_t).
\eeas 
\end{proof}
Finally, we prove Theorem \ref{thm:2ndorderODEb}. Given the proof of Theorem \ref{thm:2ndorderODE} and the Lemmas proven above, this proof is quite straightforward.
\begin{proof}
One can check as in the proof of Theorem \ref{thm:2ndorderODE} that $\frac{1}{2} \nabla (\log f_t - \log g_t) $ solves the continuity equation in the classical sense. Assumption \eqref{eq:2ndorderODE} ensures enough integrability to conclude that this vector field is indeed the velocity field of $(\mu_t)$. Using this representation of the velocity field, assumption \eqref{eq:2ndorderODEb2} also ensures that $(\mu_t)$ is a regular curve. Repeating the same calculation as in the proof of Theorem \ref{thm:2ndorderODE} we can prove that
\bes
\partial_t v_t + \frac{1}{2} \nabla |v_t|^2 = \frac{1}{8} \otgrad  \fishu .
\ees
Then, \eqref{eq:2ndorderODE} enables to apply \ref{covident} to conclude the covariant derivative $\frac{\bD}{dt}v_t$ exists and it has the desired form.
\end{proof}
\subsection{Proof of Theorem \ref{thm:entropybound}} 

\subsubsection{ First proof of (i) $\Rightarrow $ (ii)}
Let us note that the hypothesis of the Theorem imply Assumption \ref{ref:mainass}(B). Thus, we can use all the Lemmas proven in the previous section, as well as Theorem \ref{thm:2ndorderODE} and Lemma \ref{lm:compactcase} from the Appendix.
 The proof relies on the preparatory Lemmas \ref{lm:costmoto}, \ref{lm:firstder} and \ref{lm:secondder}. In the first Lemma we compute an expression for $c(\mu,\nu)$ in terms of $\tf_t,\tg_t$. 
\begin{lemma}\label{lm:costmoto}
	Let $c(\mu,\nu)$ be defined via \eqref{eq:encons1}. Then
\[\forall \, 0 < t < 1,\quad  c(\mu,\nu) =  -\int_M \,\nabla \log \tilde{f}_t \cdot \nabla \log \tilde{g}_t + 2 \scrU \, d \mu_t \]
\end{lemma}
\begin{proof}
By Corollary \ref{cor:encons},
\[ c(\mu,\nu) =\frac{1}{\sigma^2} | v_t |^2_{\tsp_{\mu_t}} - \frac{1}{4} \fish(\mu_t) \]
We can exploit Lemma \ref{lm:speed} and an integration by parts under $\vol$ (justified by Lemma \ref{lm:compactcase}) to obtain
\beas
\frac{1}{\sigma^2}| v_t |^2_{\tsp_{\mu_t}} - \frac{1}{4} \fish(\mu_t)=\int_M|  \frac{1}{2} \nabla \log \tilde{g}_t-\frac{1}{2} \nabla \log \tilde{f}_t |^2 - \frac{1}{4}| \nabla \log \mu_t + 2\nabla U |^2  d\mu_t \\
= \int_{M} \frac{1}{4}|   \nabla \log \tilde{g}_t - \nabla \log \tilde{f}_t|^2 - \frac{1}{4}|  \nabla \log \tilde{f}_t +  \nabla \log \tilde{g}_t|^2 - \nabla \log \mu_t \cdot \nabla U -|\nabla U|^2 d \mu_t \\
=\int_{M} -\nabla \log \tilde{f}_t \cdot \nabla \log \tilde{g}_t - |\nabla U|^2  d \mu_t - \int_M \nabla  \mu_t \cdot \nabla U d \vol \\
= \int_{M} -\nabla \log \tilde{f}_t \cdot \nabla \log \tilde{g}_t + \Delta U- |\nabla U|^2 d \mu_t \\
= - \int_{M} \nabla \log \tilde{f}_t \cdot \nabla \log \tilde{g}_t + 2 \scrU \, d \mu_t ,
\eeas	
which is the desired conclusion.
\end{proof}
In the second Lemma we compute the first derivative of the entropy along $\text{SB}(\mathscr{L},\mu,\nu)$. 
\begin{lemma}\label{lm:firstder}
There exist functions $\fwh$ and $\bwh$ such that:
\[ \ent(\mu_t) = \fwh(t)+ \bwh(t)  \]	and for all $0<t<1$ 
\[ \partial_t \fwh(t) = -\frac{1}{2\sigma} | v_t -\frac{\sigma}{2} \otgrad \ent(\mu_t) |^2_{\tsp_{\mu_t}}  \quad \partial_t \bwh(t) = \frac{1}{2\sigma} | v_t + \frac{\sigma}{2} \otgrad \ent(\mu_t) |^2_{\tsp_{\mu_t}}. \]
\end{lemma}
\begin{proof}
First, observe that equation \eqref{eq:fgdect} combined with Lemma \ref{lm:compactcase} make sure that all the exchanges of integral and derivatives which follow are justified and all the expressions make sense (i.e. all integrals are finite). We have
\bea\label{eq:hfhb} \ent (\mu_t) &\stackrel{\eqref{eq:ent},\eqref{eq:fgdect}}{=}&  \int_{M} \tilde{f}_t \tilde{g}_t (\log \tilde{f}_t \tilde{g}_t+2U) d\mu_t\\ \nonumber &=& \underbrace{\int_{M} \tilde{f}_t \tilde{g}_t (\log \tilde{f}_t+U) d\mu_t }_{:= \fwh(t)}+\underbrace{\int_{M} \tilde{f}_t \tilde{g}_t (\log  \tilde{g}_t+U) d\mu_t}_{:=\bwh(t)} . \eea	
We can write
\[ \fwh(t) = \frac{1}{2}\ent(\mu_t) - \mathscr{E}_{ \frac{1}{2} \log \tilde{g}_t - \frac{1}{2} \log \tilde{f}_t }(\mu_t), \]
where the definition of $\mathscr{E}_{ \frac{1}{2} \log \tilde{g}_t - \frac{1}{2} \log \tilde{f}_t }$ is given at \eqref{eq:endef}.
Therefore
\[ \partial_t \fwh (t)= \frac{1}{2}\langle \otgrad \ent(\mu_t) , v_t \rangle_{\tsp_{\mu_t}} -  \langle \frac{1}{2}\nabla(\log \tilde{g}_t -\log \tilde{f}_t),v_t\rangle_{\tsp_{\mu_t}} -\scrE_{\frac{1}{2}\partial_t
\big(\log \tilde{g}_t - \log \tilde{f}_t \big)} (\mu_t) \] 
By Lemma \ref{lm:speed} we have that $v_t = \frac{\sigma}{2}\nabla(\log \tilde{g}_t -\log \tilde{f}_t)$, so that 
\[\langle \nabla (\frac{1}{2}\log \tilde{g}_t -\frac{1}{2}\log \tilde{f}_t),v_t\rangle_{\tsp_{\mu_t}} =\frac{1}{\sigma} | v_t|^2_{\tsp_{\mu_t}}.\]
Moreover
\beas
 \scrE_{\frac{1}{2}\partial_t
\big(\log \tilde{g}_t - \log \tilde{f}_t \big)} (\mu_t) &=& \frac{1}{2}\int_M (\partial_t \log \tg_t - \partial_t \log \tf_t) d\mu_t \\ &\stackrel{\eqref{eq:fgdect}}{=}& \frac{1}{2}\int_M \partial_t \tg_t \tilde{f}_t -  \partial_t \tf_t \tg_t \, d\vol   \\
&\stackrel{ \sigma \times \eqref{eq:FKpde}}{=}& -\frac{\sigma}{4}\int_M \Delta \tilde{g}_t \tilde{f}_t + \Delta \tilde{f}_t \tilde{g}_t d\vol + \sigma \int_M 2\scrU  \tilde{f}_t \tilde{g}_t d\vol \\
&\stackrel{\text{Int. by parts}}{=}& \frac{\sigma}{2}\int_M (\nabla \log \tilde{g}_t \cdot \nabla \log \tilde{f}_t +2 \scrU ) d\mu_t \stackrel{\text{Lemma}\, \ref{lm:costmoto}}{=} -\frac{\sigma}{2} c(f,g).
 \eeas
 Therefore, we have
 \beas \partial_t \fwh (t) &=& \frac{1}{2}\langle \otgrad \ent(\mu_t) , v_t \rangle_{\tsp_{\mu_t}} - \frac{1}{\sigma}|v_t|^2_{\tsp_{\mu_t}}  +\frac{\sigma}{2} c(f,g)  \\
&
\stackrel{\text{Cor.} \,\ref{cor:encons}}{=}& \frac{1}{2}\langle \otgrad \ent(\mu_t) , v_t \rangle_{\tsp_{\mu_t}} - \frac{1}{2\sigma}|v_t|^2_{\tsp_{\mu_t}} -\frac{\sigma}{8}\fish(\mu_t) \\
 &\stackrel{\eqref{eq:fishent}}{=}&  \frac{1}{2}\langle \otgrad \ent(\mu_t) , v_t \rangle_{\tsp_{\mu_t}} - \frac{1}{2\sigma}|v_t|^2_{\tsp_{\mu_t}} -\frac{\sigma}{8} |\otgrad \ent (\mu_t)|^2_{\tsp_{\mu_t}} \\
 &=& -\frac{1}{2\sigma} | v_t - \frac{\sigma}{2}\otgrad \ent (\mu_t)|^2_{\tsp_{\mu_t}}.
 \eeas
 which is the desired conclusion. The proof of the other identity is analogous.
\end{proof}
In the last Lemma we compute the second derivative of the entropy.
\begin{lemma}\label{lm:secondder}
We have for all $0<t<1$  
\[ \partial_{tt}\fwh(t) = \frac{1}{2}  \Big\langle \hess_{\mu_t} \ent \Big(v_t-  \frac{\sigma}{2}\otgrad \ent(\mu_t)\Big), v_t- \frac{\sigma}{2} \otgrad \ent (\mu_t)\Big\rangle_{\tsp_{\mu_t}}, \]
and
\[ \partial_{tt} \bwh(t) = \frac{1}{2}  \Big\langle \hess_{\mu_t} \ent \Big(v_t+  \frac{\sigma}{2}\otgrad \ent(\mu_t)\Big), v_t+ \frac{\sigma}{2} \otgrad \ent (\mu_t)\Big\rangle_{\tsp_{\mu_t}}. \]
\end{lemma}

\begin{proof}
 Lemma \ref{lm:regcurv} ensures $(\mu_t)$ is regular. On the other hand, Lemma \ref{lm:compactcase} combined with \eqref{eq:fgdect} and \eqref{eq:entgrad} show that $\otgrad \ent(\mu_t) \in \cC^b_{\infty}$ on any $D_{\veps}$. Thus, we can use Lemma \ref{covident} to conclude that $\frac{\bD}{dt} \otgrad \ent(\mu_t)$ is well defined. 
Using Definition \ref{def:hessent}, we also obtain that
\be\label{eq:entropybound1} \frac{\bD}{dt} \otgrad \ent(\mu_t)\Big|_{t=s}= \hess_{\mu_s} \ent( v_s ) ,\ee
whereas from the compatibility with the metric we get
\be\label{eq:entropybound2} \otgrad \fish (\mu_t) \stackrel{\eqref{eq:fishent}}{=} \otgrad |\otgrad \ent |_{\tsp_{\mu_t}}^2 = 2\, \hess_{\mu_t} \ent ( \otgrad \ent ), \ee
where we wrote $\otgrad \ent $ in place of $\otgrad \ent(\mu_t)$ to simplify notation. We shall retain this convention in the next calculations. We thus have, using the compatibility with the metric
\beas
\partial_{tt}{\fwh}(t) &\stackrel{\text{Lemma}\, \ref{lm:firstder}}{=}& -\frac{1}{2\sigma} \partial_t  | v_t -\frac{\sigma}{2} \otgrad \ent |^2_{\tsp_{\mu_t}}  \\
&\stackrel{\eqref{eq:metcomp}}{=}& -\frac{1}{\sigma}\Big\langle \otcovcur \big( v_t -\frac{\sigma}{2} \otgrad \ent \big), v_t -\frac{\sigma}{2} \otgrad \ent \Big\rangle_{\tsp_{\mu_t}}
\\
&\stackrel{\text{Th.} \,\ref{thm:2ndorderODE} }{=}& -\Big\langle  \frac{\sigma}{8} \otgrad \fish(\mu_t) -\frac{1}{2} \otcovcur\otgrad \ent , v_t -\frac{\sigma}{2} \otgrad \ent \Big\rangle_{\tsp_{\mu_t}}\\
&\stackrel{\eqref{eq:entropybound1}+\eqref{eq:entropybound2}}{=}& -\Big\langle \hess_{\mu_t} \ent \big(\frac{\sigma}{4} \otgrad \ent -\frac{1}{2} v_t\big), v_t - \frac{\sigma}{2} \otgrad \ent \Big\rangle_{\tsp_{\mu_t}} \\
&=& \frac{1}{2} \Big\langle \hess_{\mu_t} \ent \big( v_t -\frac{\sigma}{2} \otgrad \ent), v_t - \frac{\sigma}{2} \otgrad \ent  \Big\rangle_{\tsp_{\mu_t}} .
\eeas	
The other identity is proven analogously.
\end{proof}
Let us now prove Theorem \ref{thm:entropybound}.
\begin{proof}
We first assume that $f$ and $g$ are continuous. The Bakry \'Emery condition \eqref{eq:Bakem} grants $\lambda$-convexity of the entropy. Therefore we get
\bes  \forall t\in(0,1), \quad \partial_{tt}\fwh(t) \stackrel{\eqref{eq:entconv}}{\geq} \frac{ \lambda}{2}| v_t -\frac{\sigma}{2} \nabla \ent |^2_{\tsp_{\mu_t}}  \stackrel{\text{Lemma} \, \ref{lm:firstder}}{=} - \lambda\sigma \partial_t \fwh (t), \ees
 where the use of \eqref{eq:entconv} is justified by the fact that $\mu_t \in \cC^+_{\infty}$ and $\otgrad \ent(\mu_t) \in \cC^{b}_{\infty}$. In the same way,
\bes
\forall t \in (0,1), \quad \partial_{tt} \bwh(t) \geq  \lambda \sigma \partial_t \bwh(t).
\ees
	Since $f,g$ are continuous, it is easy to see that both $h_f$ and $h_b$ are continuous over the whole $[0,1]$. Moreover, Lemma \ref{lm:firstder} and Lemma\ref{lm:secondder} make sure that they are $\cC^2$ over $(0,1)$. Thus, we can apply  Lemma \ref{lm:diffineq} (see the Appendix) to obtain
\[ \fwh(t) \leq \fwh(0) + (\fwh(1) - \fwh(0) ) \frac{1-\exp(-\lambda\sigma t)}{1-\exp(-\lambda\sigma)} \]
and 
\[ \bwh(t) \leq \bwh(1) + (\bwh(0) - \bwh(1) ) \frac{1-\exp(-\lambda \sigma(1-t) )}{1-\exp(-\lambda\sigma)} \]
Summing the two inequalities we obtain
\beas (1-\exp(-\lambda \sigma))\ent(\mu_t) &\leq&  h_{f}(0)\big(\exp(-\lambda\sigma t) - \exp(-\lambda\sigma) \big) \\
&+& h_{b}(0)\big(1-\exp(-\lambda\sigma(1-t))\big)\\
&+& h_{f}(1)\big(1-\exp(-\lambda\sigma t)\big) \\
&+&h_b(1)\big(\exp(-\lambda\sigma(1-t))-\exp(-\lambda\sigma)\big)\\
&=& (h_f(0)+h_{b}(0))\big(1-\exp(-\lambda\sigma(1-t))\big)\\& +& (h_f(1)+h_{b}(1))\big(1-\exp(-\lambda\sigma t)\big)\\
&-& (h_f(0)+h_{b}(1)) \big( \exp(-\lambda\sigma t) -1\big)\big( \exp(-\lambda\sigma (1-t)) -1\big).
\eeas
Dividing by $(1-\exp(-\lambda \sigma ))$ we arrive, after some simple calculations at
\bea\label{eq:entbound1}
 \ent(\mu_t) &\leq & (h_f(0)+h_{b}(0))\frac{1-\exp(-\lambda\sigma(1-t))}{1-\exp(-\lambda\sigma)}\\
 \nonumber & +& (h_f(1)+h_{b}(1))\frac{1-\exp(-\lambda\sigma t)}{1-\exp(-\lambda\sigma)}\\
\nonumber &-& (h_f(0)+h_{b}(1)) \frac{\cosh(\frac{\lambda\sigma}{2}) - \cosh(\lambda\sigma(t-\frac{1}{2}))}{\sinh(\frac{\lambda\sigma}{2})}
\eea
Observe that, by definition 
\be\label{eq:entbound2} \fwh(0) + \bwh(0) = \ent(\mu_0 )= \ent(\mu), \quad \fwh(1) + \bwh(1) = \ent(\mu_1)= \ent(\nu). \ee
From the definition of $\tf_0,\tg_0$, we obtain, using the standard properties of conditional expectation:
\beas \fwh(0)&\stackrel{\eqref{eq:hfhb}}{=}&\int_M  \tf_0 \tg_0 (\log \tf_0+U)  d \vol \\ 
&=& \int_M  f_0 g_0 \log f_0  d \inv\\
&=& E_{\bP}\Big(f(X_0)E_{\bP}[g(X_1)|X_0] \log f(X_0)\Big)\\
&=& E_{\bP}\Big(f(X_0)g(X_1)\log f(X_0)\Big).
\eeas
Using the same argument we can show that
\[ h_{b}(1) = E_{\bP}\Big(f(X_0)g(X_1)
\log g(X_1)\Big), \]
meaning that 
\be\label{eq:entbound3} h_f(0)+h_{b}(1) =  E_{\bP}\Big(f(X_0)g(X_1)
\log [ f(X_0) g(X_1) ] \, \Big) = E_{\hat{\bQ}}\Big( \log \frac{d \hat{\bQ}}{d\bP} \Big)= \entcost(\mu,\nu). \ee 
Plugging \eqref{eq:entbound2} and
\eqref{eq:entbound3} into \eqref{eq:entbound1} yields the conclusion.The general case when $f,g$ are not continuous is obtained with a standard approximation argument.
\end{proof}
\subsubsection{Second proof of (i) $\Rightarrow (ii)$}
In this proof we assume for simplicity $\sigma=1$. There is no difficulty in extending it to the general case. We recall the definitions of the $\Gamma$ and $\Gamma_2$ operators on $\cinf \times \cinf$
\be\label{eq:g1}
\Gamma(f,g) := \frac{1}{2}[ \scrL(fg) - f (\scrL g) - g (\scrL f) ],
\ee
\be\label{eq:g2}
\Gamma_2(f,g) := \frac{1}{2} [ \scrL \Gamma(f,g) - \Gamma(\scrL f,g) - \Gamma(f,\scrL g) ].
\ee
Since $M$ is compact, the integration by parts formula for the invariant measure 
$\inv $ tells that for any $f,g \in \cinf(=\cinf^b=\cinf^c)$.  
\be\label{eq:IBPF} \int_M f \scrL g d\inv = \int_M g \scrL f d\inv = - \int_M \Gamma(f,g) d\inv  , \ee
 see the monograph \cite{bakry2013analysis} for details.
The next two lemmas are the analogous to the Lemma \ref{lm:firstder} and \ref{lm:secondder}. In their proofs, we will often exchange integrals and time derivatives, and use the integration by parts formula \eqref{eq:IBPF}. All these operations are justified by Lemma \ref{lm:compactcase}, where it is proven that $f_t,g_t,\log f_t,\log g_t$ are of class $\cC^{b}_{\infty}$ over $D_{\veps}$, for any $\veps \in (0,1)$.
\begin{lemma}\label{lm:gamma1}
Let $f,g$ be given by \eqref{eq:fgdec} and $f_t,g_t$ be given by \eqref{def:fg}. 
Assume that $f$ and $g$ are of class $\cinfcp$. 
Then for any $t \in [0,1]$, we can write $\ent(\mu_t) = \fwh(t)+\bwh(t)$, where
\be\label{eq:gamma13}
\fwh(t) = \int_M \log f_t \, f_t g_t d \inv, \quad \bwh(t) = \int_M \log g_t \, f_t g_t d \inv
\ee
 Moreover, for all $0<t<1$:
\be\label{eq:gamma11}
\partial_t\fwh(t) = - \int_M  \Gamma(\log f_t, \log f_t) f_t g_t d\inv , \quad \partial_t\bwh(t) =  \int_M  \Gamma(\log g_t, \log g_t)f_t g_t d \inv
\ee
\end{lemma}
\begin{proof}
From \eqref{eq:fgdect} we get
\[ \ent(\mu_t ) = \int_M f_tg_t \log f_tg_t d \inv \]
 which yields \eqref{eq:gamma13}. Recall that $f_t,g_t$ and their logarithms are classical solutions over $[0,1] \times M$ of
\bea\label{eq:gammaheat} \partial_t f_t &=& \scrL f_t , \quad \partial_t \log f_t = \scrL \log f_t + \Gamma(\log f_t, \log f_t), \\ 
 \nonumber \partial_t g_t &=& -\scrL g_t, \quad  \partial_t \log g_t = -\scrL \log g_t - \Gamma(\log g_t, \log g_t),\eea
where we used the fact that $\Gamma(f,f) = \frac{1}{2}|\nabla f|^2$
Hence
\[
\partial_t \fwh(t) \stackrel{\eqref{eq:gammaheat}}{=} \int_M  (\scrL f_t )g_t \log f_t - (\scrL g_t) f_t \log f_t + \scrL(\log f_t ) f_t g_t + \Gamma(\log f_t,\log f_t) f_t g_t d\inv
\]
Using integration by parts , we get 
\beas &{}&\int_M  (\scrL f_t )g_t \log f_t - (\scrL g_t)  f_t\log f_t  + \scrL(\log f_t ) f_t g_t + \Gamma(\log f_t, \log f_t)f_t g_t d\inv \\
&=& \int_M  (\scrL f_t )g_t \log f_t -  g_t \scrL( f_t\log f_t)  + \scrL(\log f_t ) f_t g_t + \Gamma(\log f_t, \log f_t)f_t g_t d\inv \\
&=&  \int_M \Gamma(\log f_t, \log f_t)f_t g_t -2\Gamma(\log f_t, f_t) g_t d \inv
\\
&=& - \int_M \Gamma(\log f_t, \log f_t)f_t g_t d \inv.
\eeas
The other identity is derived analogously.
\end{proof}
\begin{lemma}\label{lm:gamma2}
In the same hypothesis of Lemma \ref{lm:gamma1} we have for all $0<t<1$
 \bea\label{eq:gamma21}
\partial_{tt} \fwh(t) &=& 2 \int_M  \Gamma_2(\log f_t, \log f_t)f_t g_td\inv, \\ \partial_{tt} \nonumber \bwh(t) &=& 2 \int_M  \Gamma_2(\log g_t, \log g_t)f_t g_t d\inv.
\eea
\end{lemma}

\begin{proof}
Using Lemma \ref{lm:gamma1} we get
\bea \label{eq:gamma22} \nonumber \partial_{tt} \fwh(t) &\stackrel{\eqref{eq:gammaheat}}{=}& - \int_M (\scrL f_t) \Gamma(\log f_t,\log f_t) g_t -(\scrL g_t) \Gamma(\log f_t,\log f_t) f_t\\
&+&  2 \Gamma(\scrL \log f_t, \log f_t)f_tg_t + 2\Gamma( \Gamma(\log f_t,\log f_t),\log f_t)f_tg_t d\inv
\eea
We can now use \eqref{eq:IBPF} and the fact that $\Gamma$ is a derivation to rewrite the last term in the above expression
\beas
&{}&2\int_M  \Gamma(\Gamma(\log f_t,\log f_t),\log f_t) f_t g_t d\inv \\&=&  2\int_M   \Gamma(\Gamma(\log f_t,\log f_t), f_t)g_t d \inv \\
&\stackrel{\eqref{eq:g1}}{=}&
\int_M  \scrL( f_t  \,\Gamma (\log f_t,\log f_t) ) g_t \\ &-&  \scrL \Gamma(\log f_,\log f_t)f_t g_t -   (\scrL f_t) \Gamma(\log f_t,\log f_t) g_td \inv  \\
&\stackrel{\eqref{eq:IBPF}}{=}& \int_M (\scrL g_t)   \Gamma (\log f_t,\log f_t) f_t -  \scrL \Gamma(\log f_t,\log f_t)\, f_t g_t\\& -&   (\scrL f_t) \Gamma(\log f_t,\log f_t)g_t d\inv.
\eeas
Plugging this expression in \eqref{eq:gamma22}, and using the definition of $\Gamma_2$, we arrive at
\beas 
\partial_{tt} \fwh(t) =  \int_M  \Big( \scrL \Gamma(\log f_t,\log f_t)  -  
2  \Gamma( \scrL \log f_t, \log f_t ) \Big) g_tf_t  d \inv \\ = 2 \int_M   \Gamma_2(\log f_t, \log f_t)f_t g_t d \inv,
\eeas
which is the desired conclusion. The other identity is proven analogously.
\end{proof}
Let us now complete the proof of Theorem \ref{thm:entropybound}.
\begin{proof}
Let $\mu,\nu$ be such that $f,g \in \cC^{b,+}_{\infty}$. It is well known that\footnote{The constant $\frac{\lambda}{2}$ instead of $\lambda$ is because the generator $\scrL$ has a $\frac{1}{2}\Delta$ as second order part instead of $\Delta$ } under \eqref{eq:Bakem} 
\[ \forall f \in \cC_{\infty} \quad \Gamma_2(f,f) \geq \frac{\lambda}{2} \Gamma(f,f). \]
In view of Lemmas \ref{lm:gamma1} and \ref{lm:gamma2}, this implies that 
\[ \partial_{tt}\fwh(t)\geq - \lambda \partial_t \fwh(t), \quad \partial_{tt}\bwh(t) \geq \lambda \partial_t \bwh(t). \] 
From this point on, the proof goes as the previous case, and we do not repeat it.  The case when $f,g$ are not both in $\cC^{b,+}_{\infty}$ follows with a standard approximation argument.
\end{proof}
\subsubsection{Proof of (ii) $\Rightarrow$ (i)}
\begin{proof}
Let us choose $\mu$ and $\nu$ such that $\mu,\nu \ll \inv$ and with bounded density, and let $(\mu^{\veps}_t)$ be the marginal flow of $SB(\mu,\nu,\veps \scrL)$. Then, equation \eqref{eq:entropybound} tells that 
\beas \nonumber \forall t \in [0,1] \quad \ent(\mu^{\veps}_t)  \leq \frac{1-\exp(- \veps\lambda(1-t))}{ 1 -\exp(-  \veps\lambda )}\ent(\mu)+ \frac{1-\exp(- \veps \lambda t)}{ 1 -\exp(- \veps \lambda )}\ent(\nu) \\
-  \frac{\cosh(\frac{\veps\lambda }{2})- \cosh(-\veps \lambda(t-\frac{1}{2}))}{\sinh(\frac{\veps\lambda }{2})} \cT^{\veps}_{\ent}(\mu,\nu).\eeas
Next, as it easy to check, our hypothesis allows us to apply the results of \cite[sec 6]{gigli2017second}. They tell that,
\begin{itemize}
\item For $t \in [0,1]$, \[\lim_{\varepsilon \rightarrow 0}\mu^{\varepsilon}_t \rightarrow \mu^0_t,\] where the limit is intented in the weak sense and $(\mu^0_t)$ is the unique constant speed geodesic between $\mu$ and $\nu$.
\item \[ \lim_{\varepsilon \rightarrow 0} \veps \, \cT^{\veps}_{\ent}(\mu,\nu) = \frac{1}{2} W^2_{2} (\mu,\nu).\]
\end{itemize}
In view of this, taking the $\liminf$ at both sides in the equation above and using the lower semicontinuity of $\ent$ yields
\be\label{eq:entconva} \ent(\mu^0_t) \leq (1-t) \ent(\mu) + t\ent(\nu) - \frac{\lambda}{2} t(1-t) W^2_{2} (\mu,\nu)  \ee
Therefore we obtained that the entropy is $\lambda$ convex along displacement interpolations, provided the initial and final measure are absolutely continuous with bounded density. It is well known that $\lambda$-convexity along \emph{all} geodesics implies the condition \eqref{eq:Bakem}. However, the proof of this fact given in \cite{von2005transport} uses only geodesics between uniform measures on balls, which are clearly among those for which we can prove \eqref{eq:entconva}. The conclusion follows.
\end{proof}
\subsubsection{Proof of Corollary \ref{cor:enttrans2}}
\begin{proof}
If $\ent(\mu)= + \infty$, there is nothing to prove since, $\entcost(\mu,\nu) = + \infty$ as well. Let $\ent(\mu)<+\infty$ and $\nu=\inv$. Then we can apply Theorem \ref{thm:entropybound}. After dividing by $(1-t)$ the bound \eqref{eq:entropybound} becomes, observing that entropy is non negative and $\cosh$ symmetric around $t=0$:
\bes	0 \leq \frac{1-\exp(- \lambda(1-t))}{1-t}\frac{1}{ 1 -\exp(-  \lambda)}\ent(\mu) 
-  \frac{\cosh(\frac{\lambda}{2})- \cosh(\frac{\lambda}{2}-\lambda(1-t))}{(1-t)\sinh(\frac{\lambda}{2})} \entcost(\mu,\nu).\ees
Letting $t \rightarrow 1$ the conclusion follows.
\end{proof}

\subsection{Proof of Theorem \ref{thm:fishdecay}}
We prove two preparatory Lemmas, and then the Theorem. The first Lemma is a rigorous proof of equation \eqref{eq:rec2}
\begin{lemma}\label{lm:covdevfish}
The vector field  $\otgrad \fish (\mu_t)$ is absolutely continuous along $(\mu_t)$. Moreover
\bea\label{eq:covdevfish1} 
 	\langle \frac{\bD}{dt}\otgrad\fishu(\mu_t),v_t\rangle_{\tsp_{\mu_t}} &=& 
2 \int_{\RD} | D v_t \cdot \nabla \log \mu_t  +  \nabla \mathbf{div}(v_t) |^2 d\mu_t  \\ \nonumber &-& 4 \int_{\RD}  \sum_{k,j=1}^d (Dv_t\cdot Dv_t)_{kj} \partial_{kj} \log \mu_t d\mu_t \\
\nonumber &+& 8\int_{\RD}  \langle v_t , \mathbf{Hess} \scrU \cdot v_t \rangle d\mu_t.
\eea
\end{lemma}
\begin{proof}
Point (i) of the assumption make sure that $(\mu_t)$ is regular. (ii) combined with Lemma \ref{covident} grant the desired absolute continuity and that 
\[ \frac{\bd}{dt} \otgrad \fishu(\mu_t) = \partial_t \otgrad \fishu(\mu_t) + D \otgrad \fishu(\mu_t)_t \cdot v_t.  \]
Since $v_t \in \tsp_{\mu_t}$ we also have, by definition of covariant derivative
\[\langle \frac{\bD}{dt}\otgrad\fishu(\mu_t),v_t\rangle_{\tsp_{\mu_t}} =\langle \frac{\bd}{dt}\otgrad\fishu(\mu_t),v_t\rangle_{\tsp_{\mu_t}}  \]
Let us now prove \eqref{eq:covdevfish1}. First, we do it for the case when $U=0$. In this case, from \eqref{eq:fishgrad}  we have $ \otgrad \fishu(\mu_t)= \alpha_t + \beta_t$ with
\[\alpha_t = -\nabla |\nabla \log \mu_t|^2, \quad \beta_t = - 2 \nabla \Delta \log \mu_t \]
From now on, we drop the dependence on $t$ both in $\mu_t$ and $v_t$ and adopt Einstein's convention for indexes.
Moreover, we abbreviate $\partial_{x_k}$ with $\partial_k$ and we are going to use, without mentioning it, the fact that $v$ is a gradient vector field, i.e. $\partial_k v^j=\partial_jv^k$. Since $\mu_t \in \cC^{+}_{\infty}$, we can use the continuity equation in the form $\partial_t \log \mu =- \divg(v)+ \nabla \log \mu \cdot v$  to get 
\beas (\partial_t \alpha_t + D \alpha_t \cdot v_t)^i &=& 
- 2 \partial_i ( \partial_{k} \log \mu \, \partial_k \partial_t  \log \mu)  - 2\partial_{i}(\partial_{kj} \log \mu \, \partial_k \log \mu) v^j \\
&=& 2  \partial_i \big(\partial_k \log \mu \,\partial_{k}(\partial_j v^j + \partial_j \log \mu \, v^j) \big) - 2\partial_{i}(\partial_{jk} \log \mu \, \partial_k \log \mu) v^j \\
&=& 2\partial_i \big(\partial_k \log \mu\,\partial_{jk} v^j + \partial_k \log \mu \, \partial_j \log \mu \,\partial_k v^j  \big)\\ &+& 2\partial_i\big(\partial_k \log \mu \, \partial_{kj} \log \mu \partial_k v^j\big) - 2\partial_{i}(\partial_{jk} \log \mu \, \partial_k \log \mu) v^j \\
&=& 2\partial_i \big(\partial_k \log \mu\,\partial_{jk} v^j + \partial_k \log \mu \, \partial_j \log \mu \,\partial_k v^j  \big)\\
&+& 2 \partial_{kj} \log \mu \partial_k \log \mu \, \partial_{i} v^j .
\eeas
Therefore
\beas
\langle \partial_t \alpha_t + D\alpha_t,v_t \rangle_{\tsp_{\mu_t}} &=& \underbrace{2\int \partial_i \big(\partial_k \log \mu\,\partial_{jk} v^j) v^i d\mu}_{:=A_1} + \underbrace{2 \int \partial_i(\partial_k \log \mu \, \partial_j \log \mu \,\partial_k v^j  \big) v^i d\mu}_{:=A.2}  \\
&+& \underbrace{2\int \partial_{kj} \log \mu \, \partial_k \log \mu \partial_{i} v^j v^i d \mu}_{:=A.3}.
\eeas
Let us now compute $\partial_t \beta_t + D \beta_t \cdot v_t $.
We have, using the continuity equation
\beas
(\partial_t \beta_t + D \beta_t \cdot v_t )^i &=& 2\partial_{ikk} (\partial_j v^j + v^j \partial_j \log \mu)-2\partial_{ijkk} \log \mu \, v^j \\
&=& 2 \partial_{ijkk} v^j +2 \partial_{i} \big( \partial_j \log \mu \,\partial_{kk}v^j+ 2 \partial_{kj}\log \mu \, \partial_k v^j  \\&+& \partial_{jkk}\log \mu \, v^j \big) -2 \partial_{ijkk} \log \mu \, v^j\\
&=& 2 \partial_{ijkk} v^j+2\partial_j \log \mu \,\partial_{ikk}v^j+2\partial_{ij} \log \mu \,\partial_{kk}v^j\\
&+& 4\partial_{ikj}\log \mu \, \partial_k v^j + 4\partial_{kj}\log \mu \, \partial_{ik} v^j \\
&+&  2\partial_{ijkk}\log \mu \, v^j+ 2\partial_{jkk}\log \mu \, \partial_i v^j -2\partial_{ijkk} \log \mu \, v^j \\
&=& 2 \partial_{ijkk} v^j+2\partial_j \log \mu \,\partial_{ikk}v^j+2\partial_{ij} \log \mu \,\partial_{kk}v^j\\
&+& 4\partial_{ikj}\log \mu \, \partial_k v^j + 4\partial_{kj}\log \mu \, \partial_{ik} v^j +  2\partial_{jkk}\log \mu \, \partial_i v^j  .
\eeas
Thus, 
\beas
\langle \partial_t \beta_t + D\beta_t,v_t \rangle_{\tsp_{\mu_t}} &=& \underbrace{2 \int \partial_{ijkk} v^j v^i d \mu}_{:=B.1} +\underbrace{2\int \partial_j \log \mu \,\partial_{ikk}v^j v^i d \mu}_{:=B.2} \\ &+&  \underbrace{2 \int \partial_{ij} \log \mu \,\partial_{kk}v^j v^i d\mu }_{:=B.3}  
+ \underbrace{4\int \partial_{ijk}\log \mu \, \partial_k v^j v^i d\mu}_{:=B.4} \\&+& \underbrace{4\int \partial_{kj}\log \mu \, \partial_{ik} v^j v^i d\mu}_{:=B.5}
 + \underbrace{2 \int \partial_{jkk}\log \mu \, \partial_i v^j v^i d\mu}_{:=B.6} 
\eeas
Finally we define 
\beas
C.1&=& \int (Dv \cdot Dv)_{ij} \, \partial_{ij} \log \mu d \mu\\
C.2&=& \int |\nabla \divg (v)|^2 d\mu \\
C.3&=& \int |D v \cdot \nabla \log \mu|^2 d\mu \\
C.4&=& \int \langle \nabla \divg (v) , Dv \cdot \nabla \log \mu \rangle d\mu,
\eeas
 where we denote by $Dv\cdot Dv$ the usual matrix product.
In the following lines we perform a series of Integration by parts, which are all justified by point (iii) of the hypothesis. We first integrate twice B.1 by parts.
\beas
B.1 &=& 2 \int \partial_{jk} v^j \partial_{ik}v^i d\mu + 2 \int \partial_{jk} v^k \partial_{i} v^i \partial_k \mu \, d\vol + 2 \int \partial_{jk} v^j \partial_{k} v^i \partial_i \mu \, d\vol \\
&+& 2 \int \partial_{jk} v^j  v^i \partial_{ik} \mu \, d\vol
\eeas
Next, we integrate the second term once again by parts and obtain
\bes
 2 \int \partial_{jk} v^k \partial_{i} v^i \partial_k \mu \, d \vol = - 2 \int \partial_{jk} v^k  v^i \partial_{ik}  \mu \, d \vol -2 \int \partial_{ijk} v^k  v^i \partial_{k}  \mu \, d \vol 
 \ees
Plugging this back, we get
\beas
B.1&=&2 \int \partial_{jk} v^j \partial_{ik}v^i d\mu  + 2 \int \partial_{jk} v^j \partial_{k} v^i \partial_i \mu \, d\vol 
-2\int \partial_{ijk} v^j  v^i \partial_{k} \mu \, d\vol \\
&=& \underbrace{2\int \partial_{jk} v^j \partial_{ik}v^i d\mu}_{=2 C.2}  + \underbrace{2 \int \partial_{jk} v^j \partial_{k} v^i \partial_i \log \mu \, d\mu}_{= 2 C.4}  -2\int \partial_{ijk} v^j  v^i \partial_{k} \log \mu \, d\mu \\
&=& 2 C.2  + 2 C.4 
\underbrace{-2\int \partial_{ijj} v^k  v^i \partial_{k} \log \mu \, d\mu}_{= - B.2 },
\eeas
where the last identity is obtained relabeling $j$ with $k$ and viceversa.
 Thus, $\sum_{i=1}^6 B.i = 2C.2+2C.4 + \sum_{i=3}^6 B.i$
 Now, let us integrate $A.3$ once by parts
 \beas
 A.3 = 2\int \partial_{kj} \log \mu \,\partial_k \mu \, \partial_{i} v^j v^i d \vol &=& \underbrace{- 2\int \partial_{jkk} \log \mu\, \partial_{i} v^j v^i d \mu}_{=-B.6} \\
 &{}& \underbrace{- 2\int \partial_{kj} \log \mu \partial_{i} v^j \partial_k v^i d \mu}_{=-2C.1} \\
 &{}&\underbrace{- 2\int \partial_{jk} \log \mu \partial_{ik} v^j v^i d \mu}_{=-\frac{1}{2}B.5}.
 \eeas  
 Thus, \be\label{eq:fishconv1}A.3 +\sum_{i=1}^6B.i = -2C.1 + 2C.2 + 2C.4 + B.3+ B.4+\frac{1}{2}B.5. \ee
 Next, we observe that, using the product's rule and exchanging $j$ and $k$, we have $A.1 = A.1.1+ B.3$, where 
 \bes
 A.1.1:= 2\int \partial_j \log \mu \,\partial_{ijk}v^k \, v^i\, d \mu 
 \ees
 Let us now turn to A.2. Using the product rule, and exchanging $j$ with $k$:
 \bes
 A.2 = \underbrace{4 \int \partial_{ij} \log \mu \partial_k \log \mu \partial_j v^k v^i \, d\mu}_{:= A.2.1} \underbrace{+2 \int \partial_{j} \log \mu \partial_k \log \mu \partial_{ij} v^k v^i \,d\mu }_{:=A.2.2}
 \ees
 Using that $\partial_k \log \mu d \mu =\partial_k \mu d \vol $, we integrate A.2.2. by parts
 \beas
 A.2.2 &=& \underbrace{- 2\int \partial_{jk} \log \mu \partial_{ij}v^k v^i d\mu}_{=-\frac{1}{2}B.5} \underbrace{-2\int \partial_{ijk}v^k \partial_j \log \mu v^i \,d \mu}_{= -A.1.1} \\
 &{}& \underbrace{ -2\int \partial_j \log \mu \partial_{ij}v^k \partial_k v^i d\mu}_{:=A.2.2.1}
 \eeas
 Therefore $A.1+A.2=B.3-\frac{1}{2}B.5 + A.2.1+A.2.2.1$. Combining this with \eqref{eq:fishconv1} we get
 \[ \sum_{i=1}^3A.i + \sum_{i=1}^6 B.i = -2C.1+2C.2+2C.4+2B.3+B.4 +A.2.1 +A.2.2.1 . \]
Let us now integrate A.2.2.1 by parts.
\beas
A.2.2.1&=& \underbrace{2 \int \partial_{ij}\log \mu \, \partial_j v^k \partial_k v^i d\mu}_{=2C.1} + \underbrace{2 \int \partial_{j}\log \mu \partial_j v^k \partial_{ik} v^i d\mu}_{=2C.4} \\
&+& \underbrace{2 \int \partial_{j}\log \mu \, \partial_j v^k \partial_k v^i \partial_i \log \mu \, d\mu}_{=2C.3}  
\eeas
Thus, $\sum_{i=1}^3A.i + \sum_{i=1}^6 B.i = 2C.2+2C.3+4C.4+2B.3+B.4 +A.2.1 $.
Lastly, we shall integrate $B.4$ by parts. We get
\beas B.4 &=& \underbrace{- 4 \int \partial_{ij}\log \mu \partial_{kk} v^j v^i d\mu}_{=-2 B.3} \underbrace{-4 \int \partial_{ij} \log \mu \partial_{k} v^j \partial_k v^i d\mu }_{= -4C.1 }\\
&{}& \underbrace{-4 \int \partial_{ij} \log \mu \partial_{k} v^j v^i \partial_k \log\mu  d\mu}_{=-A.2.1}.
\eeas
Hence, we can conclude that $\sum_{i=1}^3A.i + \sum_{i=1}^6 B.i = -4 C.1 + 2C.2+2C.3+4C.4 $. It is an easy calculation to see that this is indeed the right hand side of \eqref{eq:covdevfish1}. Since 
\[ \sum_{i=1}^3A.i + \sum_{i=1}^6 B.i= \langle\partial_t \beta_t + D\beta_t,v_t \rangle_{\tsp_{\mu_t}}+\langle\partial_t \beta_t + D\beta_t,v_t \rangle_{\tsp_{\mu_t}}= \langle\frac{\bD}{dt}\otgrad \scrI(\mu_t) ,v_t \rangle_{\tsp_{\mu_t}},\]
the Lemma is proven for the case when $U=0$. The general case follows observing that $\otgrad \fishu (\mu_t) = \otgrad \scrI(\mu_t) + \nabla \scrU $. The conclusion then follows with an easy calculation. 
\end{proof}
In the next Lemma we establish log-concavity of $\mu_t$.
\begin{lemma}\label{lm:logmix}
For all $0\leq t \leq 1$, $\mu_t$ is a log-concave measure
\end{lemma}
\begin{proof}
From \eqref{eq:bridgedec} we have:
\bes
\mu_t(\cdot) = \int_{\RD \times \RD} \bP^{xy}(X_t \in \cdot) \hat{\pi}(dxdy),
\ees
where $\bP^{xy}$ is the bridge of $\bP$ between $x$ and $y$. It has been show at \cite[Thm 2.1]{conforti2016fluctuations} that if $\scrU$ is convex, $\bP^{xy}(X_t \in \cdot) $ is a log-concave distribution for all $x,y$. Thus, because of point $(iv)$ of the hypothesis, $\mu_t$ is a log-concave mixture of log-concave measures, and is therefore log-concave itself. 
\end{proof}

We can now proceed to the proof of the Theorem.
\begin{proof}
Since the hypothesis of Theorem \ref{thm:2ndorderODEb} hold we can differentiate once in time the Fisher information to get
\[ \partial_t \fishu(\mu_t) = \langle \otgrad \fishu(\mu_t),v_t \rangle_{\tsp_{\mu_t}}. \]
Using Lemma \ref{lm:covdevfish} in combination with Theorem \ref{thm:2ndorderODEb} and the compatibility with the metric we get
\be\label{eq:fishclaim1} \partial_{tt}  \fishu(\mu_t) = \langle \frac{\bD}{dt} \otgrad\fishu(\mu_t),v_t\rangle_{\tsp_{\mu_t}} + \frac{1}{8}|\otgrad\fishu(\mu_t) |^2_{\tsp_{\mu_t}}.\ee
We know from Lemma \ref{lm:logmix} that $\mu_t$ is a log concave measure. Therefore, using Schur's product Theorem \cite[Th 7.5.3]{horn2012matrix}in the second term of the rhs of \eqref{eq:covdevfish1}  we get
\bes \langle \frac{\bD}{dt} \otgrad\fishu(\mu_t),v_t\rangle_{\tsp_{\mu_t}} \geq 8 \langle v_t, \mathbf{Hess} \scrU \cdot v_t\rangle_{\tsp_{\mu_t}} \geq 8 \alpha^2 |v_t|^2_{\tsp_{\mu_t}},\ees
since $\scrU$ is $\alpha^2$-convex.
Plugging this back into \eqref{eq:fishclaim1} we get 
\bes
\partial_{tt}  \fishu(\mu_t) \geq  \frac{1}{8}|\otgrad\fish(\mu_t) |^2_{\tsp_{\mu_t}} + 8 \alpha^2 |v_t|^2_{\tsp_{\mu_t}},
\ees
which proves the desired convexity and \eqref{eq:fishsecder}.
\end{proof}
\subsection*{Proof of Theorem \ref{thm:FK}}
\begin{proof}
From the Feynman-Kac formula we have that $f_t,g_t$ solve
\[ \partial_t f_t(x) = \frac{1}{2} \Delta f_t(x) - V f_t(x), \quad \partial_t g_t = -\frac{1}{2}\Delta g_t + Vg_t  \]
on any $D_{\veps}$. By taking logarithms, and using the positivity of $f_t,g_t$ we get
\bea\label{eq:FKHJB}\nonumber  \partial_t \log f_t(x) = \frac{1}{2} \Delta \log f_t(x) + \frac{1}{2} |\nabla \log f_t|^2- V \\  \partial_t \log g_t(x) = -\frac{1}{2} \Delta \log g_t(x) - \frac{1}{2} |\nabla \log g_t|^2+ V \eea

One can check directly as in the proof of Theorem \ref{thm:2ndorderODE} that $\frac{1}{2} \nabla (\log f_t - \log g_t) $ solves the continuity equation in the classical sense. Assumption \eqref{eq:FK2} ensures enough integrability to conclude that this vector field is actually the velocity field of $(\mu_t)$. Moreover, \eqref{eq:FK2} ensures that $(\mu_t)$ is a regular curve. If we replace \eqref{eq:HJBt} with \eqref{eq:FKHJB} and use the same argument as in the proof of Theorem \ref{thm:2ndorderODE} we get
\bes
\partial_t v_t + \frac{1}{2} \nabla |v_t|^2 = \otgrad (\frac{1}{8} \scrI + \scrE_K ).
\ees
Then, using  \eqref{eq:FK1} and Lemma \ref{covident}, we get that the covariant derivative exists and it has the desired form.
\end{proof}

\section*{Acknowledgments}
The author acknowledges support from CEMPI Lille and the University of Lille 1. He also wishes to thank Christian L\'eonard for having introduced him to the Schr\"odinger problem, and for many fruitful discussions.

\section{Appendix}
The following Lemma has been used in the proof of Theorem \ref{thm:entropybound}. Here, we denote $\dot{f},\ddot{f}$ the first and second derivatives of a function on the real line.
\begin{lemma}\label{lm:diffineq}
	Let $\phi:[0,1]\rightarrow \R$ be twice differentiable on $(0,1)$ and continuous on $ [0,1] $. 
	\begin{enumerate}[(i)]
	\item If $\ddot{\phi}_t \geq \lambda \dot{\phi}_t$ for all $t \in (0,1)$, then
	 \be\label{eq:diffineq1} \forall t \in [0,1], \quad \phi_t \leq \phi_1 + (\phi_0 - \phi_1 ) \frac{1-\exp(-\lambda (1-t) )}{1-\exp(-\lambda)}.\ee
	\item If $\ddot{\phi}_t \geq - \lambda \dot{\phi}_t$ for all $t \in (0,1)$, then
	 \[ \forall t \in [0,1], \quad \phi_t \leq \phi_0 + (\phi_1 - \phi_0 ) \frac{1-\exp(-\lambda t )}{1-\exp(-\lambda)}. \]
	\end{enumerate}
\end{lemma}
Note that the rhs of \eqref{eq:diffineq1} rewrites nicely as $ \frac{\exp(\lambda) - \exp(\lambda t)}{\exp(\lambda) -1} \phi_0 + \frac{\exp(\lambda t) - 1}{\exp(\lambda) -1} \phi_1$.
\begin{proof}
We prove only $(i)$, as $(ii)$ follows from $(i)$ with a simple time-reversal argument.
Let $g$ be the unique solution of the differential equation
\begin{align}\label{gio02}
\ddot{g}_t= \lambda \dot{g}_t,\ 0<t<1,\quad g_0=\phi_0,\ g_1=\phi_1.
\end{align}
 All we have to show is: $h:= \phi-g\le 0,$ because  a direct calculation shows that the solution of \eqref{gio02} is
$ \displaystyle{
g_t= \phi_1 + (\phi_0 - \phi_1 ) \frac{1-\exp(-\lambda (1-t) )}{1-\exp(-\lambda)}.
}$
\\
We see that  $\ddot{h}_t\ge \lambda \dot{h}_t, \ 0<t<1$ with $h_0=h_1=0.$
Considering the function  $u_t:=e ^{ -\lambda t}\dot{h}_t,$ $0\le  t\le 1,$ we have $\dot{u}_t=e^{ -\lambda t}[\ddot{h}_t-\lambda \dot{h}_t]\ge 0,$ which implies that $u$ is increasing, that is:
\begin{align}\label{gio01}
\dot{h}_t\ge \dot{h}_{t^*} e ^{ \lambda(t-t_*)},\qquad 0\le  t_*\le t\le 1.
\end{align}
Suppose ad absurdum that $h_{t_o}>0$ for some $0<t_o<1.$ As $h_0=0,$ there exists some $0< t_*\le t_o$ such that $h_{t_*}>0$ and $\dot{h}_{t_*}>0.$ In view of \eqref{gio01}, this implies that $h$ is increasing on $[t_*,1].$ In particular, $h_{1}\ge h_{t_*}>0,$ contradicting $h_1=0.$ Hence $h\le 0.$ 
\end{proof}

\subsection{Hessian of the Entropy and gradient of the Fisher information} 
\subsubsection*{Hessian of the entropy} In this paragraph we make some formal computations, whose aim is to give an explanation for equation \eqref{eq:entconv}. We assume $U=0$ for simplicity. Let $\mu \in \cinfcp$ and $\nabla \varphi \in \cinf^c $ be fixed. We consider  the constant speed geodesic $(\mu_t)$  such that $\mu_0 = \mu$ and $v_0 = \nabla \varphi$. Then, by definition
\[  \hess \ent (\nabla \varphi) = \frac{\bD}{dt} \otgrad \ent (\mu_t)\big|_{t=0}. \]
Using the identification of the covariant derivative at Lemma \ref{covident} and \eqref{eq:entgrad} we have that
\[ \hess \ent (\nabla \varphi) =  \partial_t \nabla \log \mu_t + \levciv_{v_t} \nabla \log \mu_t \big|_{t=0} \, . \]
Using the continuity equation 
\beas \partial_t \nabla \log \mu_t &=&  -\nabla (\frac{1}{\mu_t} \nabla \cdot (\mu_t v_t) ) \\
&=& -\nabla (\mathbf{div}(v_t) ) -\nabla \langle \nabla \log \mu_t,v_t\rangle .
\eeas
Evaluating at $t=0$ and using $v_0=\nabla\varphi$, we can rewrite the latter as
\[ -\nabla \Delta \varphi - \mathbf{Hess} \log \mu (\nabla \varphi) - \mathbf{Hess} \, \varphi( \nabla \log \mu). \]
Therefore, observing that $\levciv_{v_t} \nabla \log \mu_t \big|_{t=0} = \mathbf{Hess} \log \mu (\nabla \varphi)$, we arrive at 
\[\hess \ent (\nabla \varphi) =   -\nabla \Delta \varphi - \mathbf{Hess} \, \varphi( \nabla \log \mu) .\]
Hence, using an integration by parts:
\beas \langle \hess \ent (\nabla \varphi), \nabla \varphi \rangle_{\tsp_{\mu_t}} &=& - \int_M   \langle \nabla \varphi, \nabla \Delta \varphi \rangle d\mu - \int_M \langle \nabla \varphi, \mathbf{Hess} \varphi (\nabla \log \mu)\rangle d \mu \\
&=& - \int_M   \langle \nabla \varphi, \nabla \Delta \varphi \rangle d\mu - \int_M \langle \nabla  \mu, \mathbf{Hess} \varphi (\nabla \varphi)
\rangle d \vol  \\
&=& -\int_M \langle \nabla \varphi, \nabla \Delta \varphi \rangle d\mu - \frac{1}{2}\int_M \langle \nabla  \mu, \nabla |\nabla \varphi|^2 
\rangle d \vol \\
&=& \int_M   \frac{1}{2} \Delta |\nabla \varphi|^2- \langle \nabla \varphi, \nabla \Delta \varphi \rangle   d \mu .
\eeas
At this point one can use the Bochner-Weitzenb\"ock formula
\[ \frac{1}{2} \Delta |\nabla \varphi|^2 = \langle \nabla \varphi, \nabla \Delta \varphi   \rangle + |\mathbf{Hess} \varphi|^2_{\text{HS}} + \ric(\nabla \varphi, \nabla \varphi) \]
and the hypothesis \eqref{eq:Bakem} to obtain the conclusion. 
\subsubsection*{Gradient of the Fisher information} In this section, we shall make some formal computations to justify \eqref{eq:fishgrad}. As we did before, we assume $U=0$ for simplicity. Differentiating the relation \eqref{eq:fishent} and using the definition of Hessian we get
\[ \otgrad \scrI(\mu) = 2\, \hess_{\mu} \scrH (\otgrad \scrH(\mu)) \]
By the definition of Hessian 
\[ \hess_{\mu} (\otgrad \scrH(\mu) ) = \frac{\bD}{dt} \otgrad \scrH(\mu_t) \Big|_{t=0} ,\]
where $(\mu_t)$ is any regular enough curve such that $\mu_0=\mu$, $v_0=\otgrad \scrH(\mu)$.
From Lemma \ref{covident} such covariant is the projection on the space of gradient vector fields of
\bes
\partial_t \otgrad \scrH(\mu_t) + \levciv_{v_t} \otgrad \scrH(\mu_t) \ees
Using the continuity equation in the form $\partial_t \log \mu_t = - \nabla \cdot v_t - v_t \cdot \nabla \log \mu_t$, and recalling that $\otgrad \scrH(\mu_t) = \nabla \log (\mu_t)$ we arrive at
\beas \partial_t \otgrad \scrH(\mu_t) \Big|_{t=0} &=& - \nabla (\nabla \cdot v_t + v_t \cdot \nabla \log \mu_t ) \Big|_{t=0} \\
&=& - \nabla \big(\nabla \cdot (\nabla \log \mu)\big) - \nabla( |\nabla \log \mu|^2 )  \\
&=&- \nabla \Delta \log \mu - \nabla |\nabla \log \mu|^2.
\eeas
On the other hand
\[\levciv_{v_t} \otgrad \scrH(\mu_t) \Big|_{t=0} = \mathbf{Hess} \log (\mu_t)(v_t) \Big|_{t=0} =  \mathbf{Hess} \log (\mu)(\nabla \log \mu) =  \frac{1}{2} \nabla |\nabla \log \mu|^2. \]
Therefore
\[ \partial_t \otgrad \scrH(\mu_t) + \levciv_{v_t} \otgrad \scrH(\mu_t) \Big|_{t=0} =- \nabla \Delta \log \mu - \frac{1}{2}\nabla |\nabla \log \mu|^2,\]
and since the rhs of this vector field is of gradient type,
\[\otgrad \scrI (\mu) = 2\frac{\bD}{dt} \otgrad \scrH(\mu_t) \Big|_{t=0} =- 2\nabla \Delta \log \mu - \nabla |\nabla \log \mu|^2  ,\]
which is \eqref{eq:fishgrad}.

\subsection{Lemmas \ref{lm:logest} and \ref{lm:compactcase}}
These Lemmas are needed in the proof of Theorem \ref{thm:2ndorderODE} and \ref{thm:2ndorderODEb}.
\begin{lemma}\label{lm:logest}
Let Assumption \ref{ref:mainass} (B) hold. Then $\entcost(\mu,\nu)<+\infty$ and the dual representation \eqref{eq:fgdec} holds. Moreover 
\begin{enumerate}[(i)]	
\item $f$ and $g$ are compactly supported and $f_t,g_t$ globally bounded on $[0,1] \times \RD$.
\item For any $1 \leq l \leq 3 $ and $\varepsilon \in (0,1)$, there exist constants $A_{l,\varepsilon},B_{l,\varepsilon}$ such that 
\[\forall x \in M  \sup_{\substack{ \\ t \in [\varepsilon ,1-\varepsilon] } } \sup_{ \substack{v_1,..,v_l \in \RD \\ |v_1|,..,|v_l| \leq 1}}  |\partial_{v_l} \ldots \partial_{v_1} \log f_t(x)| \leq A_{l,\varepsilon}+B_{l,\varepsilon}|x|^l , \]
and the same conclusion holds replacing $f_t$ by $g_t$.
\item For any $\varepsilon \in (0,1)$ there exists a constant $A_{2,\varepsilon}$ such that
\[ \sup_{x \in\RD, t \in [\varepsilon ,  1-\varepsilon]}  \sup_{\substack{v_2,v_1 \in \RD\\|v_2|,|v_1|\leq 1}} | \partial_{v_2}\partial_{v_1} \log f_t(x) |  \leq  A_{2,\varepsilon}, \]
\end{enumerate}
and the same conclusion holds replacing $f_t$ by $g_t$.
\end{lemma}
\begin{proof}
Since all statements concerning $g$ are proven in the same way as those for $f$, we limit ourselves to prove the latter ones.
In the proof, we assume that $\sigma=1$, the proof for the general case being almost identical. The fact that $\entcost(\mu,\nu)<+\infty$ can be easily settled using point (b) in \cite[Prop. 2.5]{LeoSch}, whereas the dual representation is obtained from \cite[Th 2.8]{leonard2012schrodinger}. Let us show that $f$ is compactly supported. Observe
\bes
\frac{d \mu }{d \inv} = f(x)g_0(x).
\ees
Since $g_0 \in \cinfp$, and $\frac{d\mu}{d\inv}$ is compactly supported, $f$ must have the same support as $ \frac{d\mu}{d\inv}$. Moreover, since $g_0$ is bounded from below on the support of $f$, the fact that $\frac{d\mu}{d\inv}$ is bounded from above, implies that $f$ is bounded from above. It follows from the very definition of $f_t$ that they must be bounded as well. The proof of $(i)$ is complete. We only do the proof of $(ii)$ and $(iii)$ in the case $d=1$. This proof can be extended with no difficulty to the general case. We first make some preliminary observations.  For $\alpha$ fixed, the transition density of the Ornstein-Uhlenbeck semigroup is
\[ p_t(x,z)= \Big( \frac{\gamma(\alpha,t)}{2\pi}\Big)^{-1/2} \phi( \gamma(\alpha,t) (z- \exp(-\alpha t) x ))  \]
where 
\[ \phi(z) = \exp(-\frac{z^2}{2}), \quad \gamma(\alpha,t) = 2\frac{\alpha}{(1-\exp(- 2 \alpha t))  }. \]
The derivatives of $\phi$ can be computed using the Hermite polynomials $(H_m)_{m \geq 0}$. We have 
\[ \forall m\in \N, \quad \phi^{m} (z) =(-1)^m H_m(z) \phi(z). 
\]
Thus, we obtain the following formula for the $m$-th derivative of the transition density w.r.t. $x$:
\be\label{eq:logest3} \partial^m_x p_t(x,z) = (-1)^m \gamma(\alpha,t)^{m} \exp(-m \alpha t) H_m(z-\exp(-\alpha t) x) p_t(x,z) . \ee
Finally, observe that we can rewrite $f_t$ equivalently in the form
\be\label{eq:logest4} f_t(x) = \int_{\R} p_t(x,z)f(z)dz. \ee
Let us now prove (ii). Fix $1 \leq l \leq 3$. Using \eqref{eq:logest4}, we can write $\partial^l_x \log f_t(x)$ as a sum of finitely terms of the form
\[ f_t(x)^{-k} \prod_{j=1}^k \int f(z) \partial^{i_j}_x p_t(x,z) dz \]
where $k \leq l$ and $i_1,..,i_k$ are integers summing up to $l$. Plugging \eqref{eq:logest3} in this expression, the desired conclusion follows using the fact that $H_m$ is a polynomial of degree $m$, $f$ is compactly supported, and $\gamma(\alpha,t)$ is uniformly bounded from above and below for $t \in[\varepsilon,1-\varepsilon]$. To prove (iii), we compute explicitly $\partial^2_x \log f_t(x)$, using \eqref{eq:logest3}:
\beas \exp(-2 \alpha t) \gamma(\alpha,t)^{2}f^{-2}_t(x) \times \\\Big(\int f(z)  H_2(z-\exp(-\alpha t) x) p_t(x,z)dz\, \int f(z) p_t(x,z)dz  -\\
\Big[\int f(z)  H_1(z-\exp(-\alpha t) x) p_t(x,z)dz \Big]^2 \Big). \eeas
Using the explicit form of the first two Hermite polynomials and some standard calculations, the latter expression is seen to be equal to
\beas \exp(-2\alpha t) \gamma(\alpha,t) f_t(x)^{-2} \times  \\
\Big( \int f(z)  z^2 p_t(x,z)dz\, \int f(z) p_t(x,z)dz  -\Big[\int f(z)  z p_t(x,z)dz \Big]^2 \Big).
\eeas
The conclusion then follows using the fact that $f$ is compactly supported and that $\gamma(\alpha,t)$ is uniformly bounded from above and below for $t \in[\varepsilon,1-\varepsilon]$. 
\end{proof}
\begin{lemma}\label{lm:compactcase}
Let Assumption \ref{ref:mainass}(A) hold. Then the dual representation \ref{eq:fgdec} holds and 
\begin{enumerate}[(i)]
\item$\entcost(\mu,\nu) < \infty$ 
\item For any $\veps \in(0,1)$ $f_t,g_t,\tf_t,\tg_t$ are $\cinfcp$ over $D_{\veps}$ and $\log f_t, 
\log g_t, \log \tf_t,\log \tg_t$ are $\cinfc$ over $D_{\veps}$. 
\end{enumerate}
\end{lemma}
\begin{proof}
Let $\varphi_{\mu} = \frac{d \mu}{d \inv}$,$\varphi_{\nu} = \frac{d \nu}{d \inv}$ and define $\pi\in \Pi(\mu,\nu)$ by $\pi(x,y) = \varphi_{\mu}(x) \varphi_{\nu}(y) \inv \otimes \inv(dx dy)$. The theory of Malliavin calculus ensures that $(X_0,X_1)_{\#} \bP$ is an absolutely continuous measure on $M \times M$ with positive smooth density. Since $M$ is compact, then $(X_0,X_1)_{\#} \bP$ is equivalent to $ \inv \otimes \inv$. Therefore, for some constant $C<+\infty$,
\[ \scrH(\pi | (X_0,X_1)_{\#} \bP) \leq C \, \scrH(\pi | \inv \otimes \inv) = C ( \scrH(\mu | \inv)+\scrH(\nu | \inv))  <+\infty. \]
Thus $\entcost(\mu,\nu)<+\infty$. It is also a result of Malliavin calculus that $f_t,g_t$ are of class $\cC^{+}_{\infty}$ on any $D_{\veps}$. But then, since $M$ is compact, they are also in $\cC^{b,+}_{\infty}$ and uniformly bounded from below, which gives that $\log f_t,\log g_t $ are in $\cC^{b}_{\infty}$. The statement about $\tf_t,\tg_t$ and their logarithms follows from the one for $f_t,g_t$ and the compactness of $M$.
\end{proof}
\subsection*{Proof of Lemma \ref{lm:speed}}

\begin{proof}
We can rewrite \eqref{eq:fgdec} as
\[\frac{d \hQ}{d \bP} = f(X_0) g(X_1). \]
Moreover, since $\bP$ is stationary, we have for any $t$ that ${X_t}_{\#} \bP = \inv$. Therefore
\beas \frac{d\mu_t}{d\inv}(x) = \frac{d({X_t}_{\#} \hQ)}{d({X_t}_{\#} \bP)}(x) &=& E_{\bP}[ f(X_0)g(X_1) | X_t =x ] \\ &\stackrel{\substack{\text{Markov}\\ \text{property}}}{=}& E_{\bP}[f(X_0) |X_t=x] E_{\bP}[g(X_1) |X_t=x] \\
&=& f_t(x)g_t(x).  \eeas
Observing that $d\inv=\exp(-2U) d \vol $, \eqref{eq:fgdect} follows from the definition of $\tf_t$ and $\tg_t$.
To prove the second statement, fix $\varepsilon \in (0,1)$. We observe that, since $U$ is taken to be smooth, the well known results of Malliavin calculus grant that the function $f_t$ and $g_t$ are of class $\cinfp$, and thus classical solutions on $D_{\varepsilon}$ of the forward and backward Kolmogorov equations
\be\label{eq:Kolmogorov} \partial_t f_t = \scrL f_t, \quad \partial_t g_t = - \scrL g_t. \ee
Using some standard algebraic manipulations and the positivity of $f_t,g_t$ one finds that $ \tf_t$ and $ \tg_t$ are classical solutions on $D_{\varepsilon}$ of
\be\label{eq:FKpde} \partial_t \tilde{f}_t = \frac{1}{2} \Delta \tilde{f}_t  - \scrU \tilde{f}_t  , \quad \partial_t \tilde{g}_t =- \frac{1}{2} \Delta \tilde{g}_t + \scrU \tilde{g}_t  =0,  \ee
where $\scrU$ was defined at \eqref{eq:fishgrad}. 
Using this, we prove that $\frac{1}{2}\nabla (\log g_t-\log f_t) = \frac{1}{2}\nabla (\log \tg_t-\log \tf_t)$ is a classical solution to the continuity equation on $D_{\varepsilon}$. Indeed 
 \beas \partial_t \mu_t&\stackrel{\eqref{eq:fgdect}}{=}&  (\partial_t \tilde{f}_t) \tilde{g}_t + \tilde{f}_t (\partial_t \tilde{g}_t) \\
 &\stackrel{\eqref{eq:FKpde}}{=}& \frac{1}{2} \tilde{g}_t \Delta \tilde{f}_t - \frac{1}{2} \tilde{f}_t \Delta \tilde{g}_t \\
 &=& \frac{1}{2} \nabla \cdot (\tilde{g}_t \nabla \tilde{f}_t) - \frac{1}{2} \nabla \cdot (\tilde{f}_t \nabla \tilde{g}_t) \\
 &=& \frac{1}{2} \nabla \cdot \big(\tilde{f}_t\tilde{g}_t( \nabla \log \tilde{f}_t -\nabla \log \tilde{g}_t)\big) \\
 &\stackrel{\eqref{eq:fgdect}}{=}& \frac{1}{2} \nabla \cdot \big(\mu_t( \nabla \log \tilde{f}_t -\nabla \log \tilde{g}_t)\big).
 \eeas
Thus, $(t,x)\mapsto \frac{1}{2}\nabla (\log g_t-\log f_t)$ solves the continuity equation,  it is of gradient type and, thanks to Lemma \ref{lm:speedbound} and \eqref{eq:fgdect},  $ \sup_{t \in [\varepsilon,1-\varepsilon]}\frac{1}{2}| \nabla (\log g_t-\log f_t)|_{\tsp_{\mu_t}} < + \infty$ also holds. The conclusion then follows.
\end{proof}

\subsection*{Proof of Lemma \ref{covident}}
\begin{proof}
Fix $\veps \in (0,1)$. As a preliminary step, we compute $\partial_t \tau^{\veps}_t(\xi_t)$. Using the group property we get
\beas  \tau^{\veps}_{t+h}(\xi_{t+h})-\tau^{\veps}_t(\xi_t) &=& \tau^{\veps}_t \big( \tau^t_{t+h}(\xi_{t+h}) - \xi_t  \big) \\
&=& h \tau^{\veps}_t \big(  \partial_t \xi_t \big)+ \tau^{\veps}_t \big(\tau^t_{t+h}(\xi_{t})- \xi_t  \big) + o(h),
\eeas
where $o(h)/h \rightarrow 0$ as $h \rightarrow 0$. Recalling Definition \ref{def:trmap2} and the definition of flow map we get
\beas
\tau^t_{t+h}(\xi_{t})(x)- \xi_t (x) &=& (\tau_x)^t_{t+h}\big(\xi_t \circ \bT(t,t+h,x)-\xi_t(x)\big)\\
&=& h \levciv_{\partial_t \bT(t,t,x)} \xi_t(x) +o(h) \\
&=&h \levciv_{v_t(x)} \xi_t(x) +o(h) .
\eeas
Therefore, we have shown that, as a pointwise limit
\be\label{eq:covident4} 
\lim_{h \rightarrow 0} \frac{ \tau^t_{t+h} \xi_{t+h} - \xi_t}{h} = \partial_t \xi_t + \levciv_{v_t} \xi_t, \ee
which implies that
\bes 
\partial_t \tau^{\veps}_t(\xi_t) =  \tau^{\veps}_t(\partial_t \xi_t + \levciv_{v_t} \xi_t). \ees
Let us now prove the absolute continuity of $(\xi_t)$ along $(\mu_t)$ using what we have just shown. We have
\beas | \tau^{\veps}_s (\xi_s)  - \tau^{\veps}_t (\xi_t) |_{L^2_{\mu_{\veps}}} &=&\Big( \int_{M} \Big| \int_t^s \tau^{\veps}_r \big( \partial_r \xi_r + \levciv_{v_r} \xi_r  \big)  dr\Big|^2 d \mu_{\veps}\Big)^{\frac{1}{2}}\\
&\stackrel{\text{Jensen}}{\leq}&  (s-t)^{1/2} \Big(  \int_t^s \int_{M} \big| \tau^{\veps}_r \big( \partial_r \xi_r + \levciv_{v_r} \xi_r) \big|^2  d \mu_{\veps} dr\Big)^{\frac{1}{2}} \\
&\stackrel{\eqref{eq:isometry}}{=}& (s-t)^{1/2} \Big(  \int_t^s \int_{M} \big|  \partial_r \xi_r + \frac{1}{2}\levciv_{v_r} \xi_r \big|^2 d \mu_r dr\Big)^{\frac{1}{2}} \\
&\leq & (s-t) \sup_{r \in [\varepsilon,1-\varepsilon]} |\partial_r \xi_r + \levciv_{v_r} \xi_r |_{\tsp_{\mu_r}} .
  \eeas
Using \eqref{eq:covident1}, the desired absolute continuity follows. Let us now turn to the proof of \eqref{eq:covident2}.
 By definition,
\[\frac{\bd}{dt} \xi_t = \lim_{h \downarrow 0} \frac{\tau^t_{t+h}\xi_{t+h} (\cdot)- \xi_t(\cdot)}{h} , \]
where the limit is in $L^2_{\mu_t}$. 
 But then, it is also the pointwise limit along a subsequence. Such computation has been done at \eqref{eq:covident4}, and yields the desired result. The identity \eqref{eq:covident3} is a direct consequence of \eqref{eq:covident2} and the fact that $v_t$ is a gradient vector field.
\end{proof}

\end{document}